\documentclass[11pt,a4paper]{amsart}
\usepackage{amssymb}
\usepackage{amscd}
\usepackage{amsmath,amsfonts,stmaryrd}
\usepackage[usenames,dvipsnames,svgnames]{xcolor}

\usepackage{silence}
\usepackage[colorlinks, urlcolor=Navy, linkcolor=Navy, citecolor=Navy]{hyperref}
\usepackage{microtype}

\usepackage[abbrev,nobysame]{amsrefs}
\renewcommand{\MR}[1]{}

\usepackage{tikz}
\usepackage{tikz-cd}
\usetikzlibrary{arrows}
\pgfarrowsdeclarecombine{twohead}{twohead}
{angle 90}{angle 90}{angle 90}{angle 90}

\theoremstyle{plain}
\newtheorem{thm}{Theorem}[section]
\newtheorem{lem}[thm]{Lemma}

\newtheorem{pro}[thm]{Proposition}
\newtheorem{cor}[thm]{Corollary}
\newtheorem{thm*}{Theorem}
\newtheorem{lem*}[thm*]{Lemma}
\newtheorem{pro*}[thm*]{Proposition}
\theoremstyle{definition}
\newtheorem{dfn}[thm]{Definition}
\newtheorem{exa}[thm]{Example}
\newtheorem{rem}[thm]{Remark}

\numberwithin{equation}{section}

\DeclareMathOperator{\Hom}{Hom}

\DeclareMathOperator{\End}{End}
\DeclareMathOperator{\Ext}{Ext}

\DeclareMathOperator{\proj}{proj}

\DeclareMathOperator{\Spec}{Spec}

\DeclareMathOperator{\thick}{thick}
\DeclareMathOperator{\ann}{ann}

\DeclareMathOperator{\modu}{mod}

\DeclareMathOperator{\Mat}{Mat}

\newcommand{\lmod}[1]{\text{$#1$-$\modu$}}

\newcommand{\proje}[1]{\text{$#1$-$\mathrm{proj}$}}
\newcommand{\injec}[1]{\text{$#1$-$\mathrm{inj}$}}

\newcommand{\rep}[2]{R_{#1}{\left(#2\right)}}
\newcommand{\rankvar}[2]{\mathcal{D}_{#1}^{#2}}
\newcommand{\drankvar}[2]{\mathcal{C}_{#1}^{#2}}

\newcommand{\isoto}{\xrightarrow{\sim}}

\newcommand{\pr}{\mathrm{pr}}

\newcommand{\Gl}{\mathbf{Gl}}

\DeclareMathOperator{\id}{id}
\DeclareMathOperator{\idim}{id}
\DeclareMathOperator{\pdim}{pd}

\DeclareMathOperator{\Dim}{\rm \underline{dim}}

\DeclareMathOperator{\rk}{rk}

\newcommand{\oTo}{\xymatrix{ \ar@{^{(}->}[r]|{\mathbf{O}}& }} 
\newcommand{\cTo}{\xymatrix{ \ar@{^{(}->}[r]|{\mathbf{|}}& }} 
\newcommand{\coTo}{\xymatrix{ \ar@{^{(}->}[r]|{\mathbf{O}}|{\mathbf{|}}& }} 

\newcommand{\op}{\mathrm{op}}
\newcommand{\st}{\mathrm{st}}
\newcommand{\cost}{\mathrm{cst}}
\DeclareMathOperator{\Ker}{ker}
\DeclareMathOperator{\coKer}{coker}
\DeclareMathOperator{\Bild}{im}

\DeclareMathOperator{\rad}{rad }

\DeclareMathOperator{\add}{add}
\DeclareMathOperator{\gldim}{gldim}

\newcommand{\kdual}{\mathrm{D}}

\newcommand{\D}{\mathbb{D}}

\newcommand{\K}{\mathbb{K}}

\newcommand{\N}{\mathbb{N}}
\renewcommand{\P}{\mathbb{P}}

\newcommand{\Z}{\mathbb{Z}}

\newcommand{\mcC}{\mathcal{C}}

\newcommand{\mcH}{\mathcal{H}}

\newcommand{\mcK}{\mathcal{K}}

\newcommand{\mcO}{\mathcal{O}}

\newcommand{\mcQ}{\mathcal{Q}}

\newcommand{\mcS}{\mathcal{S}}
\newcommand{\mcT}{\mathcal{T}}

\DeclareMathOperator{\gen}{gen}
\DeclareMathOperator{\cogen}{cogen}

\DeclareMathOperator{\Gr}{Gr}

\newcommand{\Gra}[3]{\mathrm{Gr}_{#1}{\tbinom{#2}{#3}}}

\newcommand{\orbit}[1]{\mathcal{O}_{#1}}
\newcommand{\clorbit}[1]{\overline{\mathcal{O}}_{#1}}
\newcommand{\grstrat}[1]{\mathcal{E}^{[#1]}}
\newcommand{\clgrstrat}[1]{\overline{\mathcal{E}}^{[#1]}}

\newcommand{\gitquot}{/\hspace{-4pt}/}

\newcommand{\homot}[2][]{\mathcal{K}^{#1}(#2)}
\newcommand{\dercat}[2][]{\mathcal{D}^{#1}(#2)}
\newcommand{\bound}{\mathrm{b}}

\newcommand{\tang}[2]{\mathrm{T}_{#1}{#2}}

\usepackage{todonotes}

\renewcommand{\epsilon}{\varepsilon}

\renewcommand{\tilde}{\widetilde}

\newcommand{\tiltmod}[1]{T_{#1}}
\newcommand{\cotiltmod}[1]{C^{#1}}
\newcommand{\tiltalg}[1]{B_{#1}}
\newcommand{\cotiltalg}[1]{B^{#1}}

\begin{document}

\title[On quiver Grassmannians and orbit closures for gen-finite modules]{On quiver Grassmannians and orbit closures\\for gen-finite modules}

\author{Matthew Pressland}
\address{Matthew Pressland\\Institut f\"ur Algebra und Zahlentheorie\\Universit\"at Stuttgart\\Pfaffenwaldring~57\\70569 Stuttgart\\Germany}
\email{presslmw@mathematik.uni-stuttgart.de}

\author{Julia Sauter}
\address{Julia Sauter\\Fakult\"at f\"ur Mathematik\\Universit\"at Bielefeld\\Postfach 100 131\\D-33501 Bielefeld\\Germany}
\email{jsauter@math.uni-bielefeld.de}

%
%



\begin{abstract} 
We show that endomorphism rings of cogenerators in the module category of a finite-dimensional algebra $A$ admit a canonical tilting module, whose tilted algebra $B$ is related to $A$ by a recollement. Let $M$ be a gen-finite $A$-module, meaning there are only finitely many indecomposable modules generated by $M$. Using the canonical tilts of endomorphism algebras of suitable cogenerators associated to $M$, and the resulting recollements, we construct desingularisations of the orbit closure and quiver Grassmannians of $M$, thus generalising all results from previous work of Crawley-Boevey and the second author in 2017. We provide dual versions of the key results, in order to also treat cogen-finite modules.
\end{abstract}
\maketitle

\section{Introduction}

A celebrated result of Hironaka \cite{Hir} states that any scheme $X$ over a field $\K$ of characteristic zero admits a desingularisation, meaning a map $f\colon Y\to X$ of schemes such that $Y$ is smooth and $f$ is an isomorphism over the non-singular locus of $X$. When $\K$ is algebraically closed and $X$ is projective, it follows from the work of several authors \cites{H1,H2,HZ,BHZ,DHZW,R3,Ri4,Ri5} that $X$ may be realised representation-theoretically as a quiver Grassmannian. Given a finite-dimensional basic algebra $A$ (which we realise as an admissible quotient of the path algebra of a quiver $Q$), an $A$-module $M$, and a dimension vector $d\in\N^{Q_0}$, the quiver Grassmannian $\Gra{A}{M}{d}$ is a projective variety parametrising the $A$-submodules of $M$ with dimension vector $d$. Given such a realisation $X\cong\Gra{A}{M}{d}$, our aim in this paper is to construct a desingularisation $f\colon Y\to X$ representation-theoretically from $A$, $M$ and $d$. We will achieve this aim under a certain condition on $M$, explained in the following paragraph, and the variety $Y$ will turn out to be a union of strata of other quiver Grassmannians.

A desingularisation of $\Gra{A}{M}{d}$ is constructed by Cerulli Irelli, Feigin and Reineke \cite{CIFR} in the case that $A$ is the path algebra of a Dynkin quiver, by Keller and Scherotzke \cite{KS2} for $A$ in a larger class including iterated tilted algebras, and by Crawley-Boevey and the second author \cite{CBSa} when $A$ is any representation-finite algebra. In this paper we work with arbitrary finite-dimensional algebras, instead placing a suitable finiteness condition on the module $M$ to be able to construct our desingularisation. Precisely, we ask that $M$ is \emph{gen-finite}, meaning that the category $\gen(M)$, consisting of quotient modules of finite direct sums of copies of $M$, has finitely many indecomposable objects up to isomorphism. If $A$ is representation-finite, as in \cites{CIFR,CBSa}, then every $A$-module is gen-finite, but there are also many more examples, such as the preinjective modules over any finite-dimensional hereditary $\K$-algebra.

We now sketch our construction, beginning with some generalities. If $B$ is a finite-dimensional algebra, and $e\in B$ is an idempotent such that $eBe=A$, then restricting $B$-modules to $A$-modules via the functor $L\mapsto eL$, which we also denote by $e$, provides an algebraic map $\Gra{B}{\tilde{M}}{d,s}\to\Gra{A}{M}{d}$ for any $s$ and any $B$-module $\tilde{M}$ with $e\tilde{M}=M$; we say here that a $B$-module $L$ has dimension vector $(d,s)$ if the dimension vector of $eL$ is $d$ and that of $(1-e)L$ is $s$. Our aim is to choose $B$ and $\tilde{M}$ carefully so that the quiver Grassmannian $\Gra{B}{\tilde{M}}{d,s}$ and the map $e$ have good properties.

When $M$ is gen-finite, we may take $E$ to be the minimal cogenerating $A$-module (with respect to number of indecomposable summands) such that every indecomposable object of $\gen(M)$ is isomorphic to a summand of $E$. In Section~\ref{s:special-tilt}, we explain how to construct an algebra $B$, derived equivalent to $\End_A(E)^{\op}$, together with an idempotent $e\in B$ such that $eBe$ is naturally isomorphic to $A$, so our first requirement is met. We call $B$ the \emph{cogenerator-tilted algebra} of $E$, and its construction generalises methodology of \cites{CIFR,KS2,CBSa} for the case that $A$ is representation-finite and $E$ is an additive generator of $\lmod{A}$, so that $\End_A(E)^{\op}$ is the  Auslander algebra of $A$; in our more general setting, the algebra $\End_A(E)^{\op}$ is typically not an Auslander algebra, and so is not so homologically well-behaved. In particular, whereas in the setting of \cites{CIFR,KS2,CBSa} the global dimension of $B$ never exceeds $2$, in our case it can be larger, requiring us to adapt our methods accordingly.

To construct the required $B$-module $\tilde{M}$ restricting to $M$, it is natural to consider functors $\varphi\colon \lmod{A}\to\lmod{B}$ such that $e\varphi$ is naturally isomorphic to the identity on $\lmod{A}$. As is the case for any idempotent, $e$ comes with two canonical such functors, given by the left and right adjoints, $\ell$ and $r$, of the restriction functor $\lmod{B}\to\lmod{A}$ along $e$. Following the approach of \cites{CIFR,KS2,CBSa}, we will not in fact use $\ell$ or $r$ to obtain $\tilde{M}$, but rather a third functor $c$, the intermediate extension \cite{Ku}, which is the image of a canonical natural transformation $\ell\to r$. This functor is better behaved than $\ell$ or $r$ when it comes to preserving rigidity properties, and we take $\tilde{M}=c(M)$; this module is readily computable in examples.

Now we are ready to describe our desingularisation of $\Gra{A}{M}{d}$ in terms of quiver Grassmannians of the $B$-module $c(M)$. Since $M$ is gen-finite, the irreducible components of $\Gra{A}{M}{d}$ are given by closed subsets $\clgrstrat{N_1},\dotsc,\clgrstrat{N_t}$ for some $N_i\in\lmod{A}$, where $\grstrat{N}\subseteq\Gra{A}{M}{d}$ is the subset of the Grassmannian consisting of submodules $U\leq M$ with $M/U\cong N$. Recalling our notation for dimension vectors of $B$, write $(d,s_i)$ for the dimension vector $\Dim{c(M)}-\Dim{c(N_i)}$. Thus we get a map $\Gra{B}{c(M)}{d,s_i}\to\Gra{A}{M}{d}$ for each $i$. Our main result is then the following.

\begin{thm*}[cf.~Corollary~\ref{Gr-desing}]
Working over an algebraically closed field $\K$, consider the maps $\Gra{B}{c(M)}{d,s_i}\to\Gra{A}{M}{d}$ constructed in the preceding paragraphs. Each such map restricts to a map $\clgrstrat{c(N_i)}\to\clgrstrat{N_i}$, and we may take the disjoint union of these to get a map
\[p\colon\bigsqcup_{i=1}^t\clgrstrat{c(N_i)}\to\bigcup_{i=1}^t\clgrstrat{N_i}=\Gra{A}{M}{d},\]
which is a desingularisation.
\end{thm*}

The key step in proving this theorem is provided by Theorem~\ref{Gr-smooth}, in which we show, using properties of the cogenerator-tilted algebra $B$ and the intermediate extension functor $c$, that the quiver Grassmannians $\Gra{B}{c(M)}{d,s_i}$ are smooth.

We note that, given a projective $\K$-variety $X$, there can be many different ways to realise $X$ as a quiver Grassmannian, so it is possible to put heavy restrictions on the triple $(A,M,d)$ while preserving the property that the quiver Grassmannians $\Gra{A}{M}{d}$ run over all projective varieties. For example, a straightforward observation is that one may restrict to the case that $A$ is a hereditary algebra---if $A=\K Q/I$ for a quiver $Q$ and admissible ideal $I$, we can think of $M$ and all of its submodules as $\K Q$-modules, all of which will automatically satisfy the relations in $I$, so $\Gra{A}{M}{d}=\Gra{\K Q}{M}{d}$. Reineke \cite{R3} has shown that one may assume that $M$ is a brick, i.e.\ that $\End_A(M)^\op=\K$, and Ringel \cite{Ri4} has shown that one may restrict to $A=\K Q$ for $Q$ an $n$-Kronecker quiver for some $n\geq 3$ (this being the quiver with vertex set $\{1,2\}$ and $n$ arrows from $1$ to $2$), $M$ a module having no injective direct summands, and $d=(1,1)$. Very recently, Ringel \cite{Ri6} has shown that it is even possible to fix $A$ completely, by choosing it to be the path algebra of a wild acyclic quiver. (On the other hand, the corresponding statement for more general wild algebras is not true \cite[\S4]{Ri5}.)

Thus, from this point of view, it is not clear whether or not our increased algebraic generality allows us to construct desingularisations of any more projective varieties when compared to prior work---to the best of our knowledge, the same is true when comparing the different algebraic generalities of results already in the literature. However, since our desingularisations, like those of \cites{CIFR,KS2,CBSa}, depend on the particular choice of $A$, $M$ and $d$ used to realise a given projective variety, we believe there is still value in enlarging the range from which this representation-theoretic data may be chosen. 

We also, analogous to \cite{CBSa}, construct desingularisations of the orbit closure of a gen-finite module $M$ over an arbitrary finite-dimensional algebra $A$, using similar techniques as for quiver Grassmannians. Desingularisations for these varieties have been constructed previously by Zwara \cite{Zw2}, but using rather different methods, and with a very different description. Denoting by $\rep{A}{d}$ the affine space of $A$-modules with dimension vector $d$, which carries an action of the product $\Gl_d=\prod_{i=1}^n\Gl_{d_i}(\K)$ of general linear groups, we denote by $\clorbit{M}$ the closure of the $\Gl_d$-orbit of $M$, taken in the Zariski topology, and it is this variety we now aim to desingularise.

Construct the cogenerator-tilted algebra $B$, and the idempotent $e$ with $eBe=A$, exactly as in the case of quiver Grassmannians (although in this case we may relax the minimality assumption on the cogenerator $E$, giving us more flexibility in the construction). Then $e\colon L\mapsto eL$ is an algebraic map $\rep{B}{d,s}\to\rep{A}{d}$ for any dimension vector $(d,s)$ for $B$. The natural group action on $\rep{B}{d,s}$ is by $\Gl_d\times\Gl_s$, and $e$ is constant on $\Gl_s$-orbits. We consider again the intermediate extension $c(M)$, choosing $(d,s)$ to be its dimension vector, and take its orbit closure $\clorbit{c(M)}\subseteq\rep{B}{d,s}$. We denote by $\clorbit{c(M)}^{\st}$ the set of stable $B$-modules in $\clorbit{c(M)}$, which in this case means those modules admitting a monomorphism to a direct sum of copies of the injective $B$-module $\Hom_\K(eB,\K)$. Our theorem is then as follows.

\begin{thm*}[Theorem~\ref{gen-finite-desing}] Let $\K$ be an algebraically closed field of characteristic zero. 
Let $A$ be a finite-dimensional algebra and $M$ a gen-finite $A$-module. Let $B$ be the cogenerator-tilted algebra of any cogenerating $A$-module $E$ such that every indecomposable quotient of $M$ is isomorphic to a direct summand of $E$. Then the distinguished idempotent $e\in B$ with $eBe=A$ induces a map
\[\pi_M\colon\clorbit{c(M)}^\st/\Gl_s\to\clorbit{M},\]
which is a desingularisation with connected fibres.
\end{thm*}

A key step in the proof, as in that of \cite{CBSa}, is to realise the orbit closure $\clorbit{M}$ as an affine quotient variety.

\begin{thm*}[cf.\ Theorem~\ref{affine-quot}]
\label{affine-quot-intro}
In the setting of Theorem~\ref{gen-finite-desing}, the distinguished idempotent $e\in B$ induces an isomorphism
\[\rep{B}{d,s}\gitquot\Gl_s\isoto\clorbit{M},\]
where $(d,s)$ is the dimension vector of $c(M)$.
\end{thm*}

In fact, Theorem~\ref{affine-quot} is a much more general statement than that given here, and states that any rank variety for $A$ (the orbit closures of gen-finite modules being a special case of such) may be realised as an affine quotient variety $\rep{B}{d,s}\gitquot\Gl_s$, for $B$ the cogenerator-tilted algebra of a suitable cogenerating $A$-module. Some of the ideas here date back to Kraft and Procesi \cite[\S3]{KP}, who obtain Theorem~\ref{affine-quot-intro} in the case that $A=\K[t]/(t^n)$ is the the truncated polynomial ring and $E$ is the basic additive generator of $\lmod{A}$. We provide further details concerning this case in Example~\ref{KP-eg}.

The structure of the paper is as follows. The construction of cogenerator-tilted algebras is given in Section~\ref{s:special-tilt}, and homological properties of these algebras determined in this and the following section. We then give the details of our desingularisation constructions, beginning with the slightly easier case of orbit closures in Section~\ref{orbit-closures}, and continuing with quiver Grassmannians in Section~\ref{quiver-Grassmannians}. Since our results apply to varieties defined by gen-finite modules, we briefly recall some constructions and properties of such modules in Section~\ref{s:gen-finite}, before closing in Section~\ref{s:example} with explicit examples in which $A$ is the path algebra of the $n$-subspace quiver, typically having wild representation-type.

Our results were announced by the second author in 2017 at the 50\textup{th} Symposium on Ring Theory and Representation Theory at the University of Yamanishi, and a summary (with the same title as this article) may be found in the associated proceedings volume \cite{PS2-proc}.

Throughout the paper, all algebras are finite-dimensional $\K$-algebras over some field $\K$ (assumed in Section~\ref{quiver-Grassmannians} to be algebraically closed and in Section~\ref{orbit-closures} to be algebraically closed and of characteristic zero), and without additional qualification `module' is taken to mean `finitely-generated left module'. We write $\kdual=\Hom_\K(-,\K)$ throughout for the standard duality over the ground field. Morphisms are composed from right-to-left.

\section{Special (co)tilting}
\label{s:special-tilt}

The goal of this section is to characterise certain tilting (and cotilting) modules which are generated by a projective summand (or respectively cogenerated by an injective summand). These modules will form the basis of our subsequent constructions. This section and the next, being purely homological, require no additional assumptions on the field $\K$.

\begin{dfn} \label{dfn-tilting}
Let $\Gamma$ be a finite-dimensional $\K$-algebra. Recall that $T\in\lmod{\Gamma}$ is a \emph{tilting module} (or sometimes \emph{classical tilting module}) if
\begin{itemize}
\item[(T1)] $\pdim{T}\leq1$,
\item[(T2)] $\Ext^1_\Gamma(T,T)=0$, and
\item[(T3)] there is an exact sequence $\begin{tikzcd}[column sep=20pt]0\arrow{r}&\Gamma\arrow{r}&T_0\arrow{r}&T_1\arrow{r}&0\end{tikzcd}$ with $T_i \in \add{T}$.
\end{itemize}
We say that $T$ is \emph{$P$-special}, for some projective $\Gamma$-module $P$, if $P\in\add{T}$ and the module $T_0$ in (T3) can be chosen to lie in $\add{P}$. Dually, $C\in\lmod{\Gamma}$ is a \emph{cotilting module} if 
\begin{itemize}
\item[(C1)] $\idim{C}\leq1$,
\item[(C2)] $\Ext^1_\Gamma(C,C) =0$, and
\item[(C3)] there an exact sequence $\begin{tikzcd}[column sep=20pt]
0\arrow{r}&C^1\arrow{r}&C^0\arrow{r}&\kdual\Gamma\arrow{r}&0
\end{tikzcd}$ with $C^i \in\add{C}$,
\end{itemize}
and we say $C$ is $Q$-special, for some injective $\Gamma$-module $Q$, if $Q\in\add{C}$ and $C^0$ in (C3) can be chosen to lie in $\add{Q}$. We say that a tilting module is \emph{special} if it is $P$-special for some $P$, and define \emph{special} cotilting modules analogously.
\end{dfn}

\begin{pro}
\label{p:uniqueness}
If $T$ and $T'$ are $P$-special tilting modules, then $\add{T}=\add{T'}$. In particular, any two basic $P$-special tilting modules are isomorphic. The analogous results hold for $Q$-special cotilting modules.
\end{pro}
\begin{proof}
Since the middle term in the exact sequence from (T3) may be chosen to lie in $\add{P}$ in both cases, we have $\gen(T)=\gen(P)=\gen(T')$, so write $\mcT$ for this subcategory. The sequence in (T3) also provides a monomorphism from $\Gamma$ to an object of $\add{P}$, so $P\in\mcT$ is faithful and $\ann(\mcT)=0$. Hence by \cite{Sma}, the direct sum $T_0$ of indecomposable Ext-projectives in $\mcT$ is a tilting $\Gamma$-module. But the tilting modules $T$ and $T'$ are also Ext-projective in $\mcT$. Since any two tilting modules have the same number of pairwise non-isomorphic direct summands, it follows that $\add{T}=\add{T_0}=\add{T'}$. The statement for cotilting modules is proved dually.
\end{proof}

\begin{dfn}
\label{special-tilt-notation}
Let $\Gamma$ be a finite-dimensional algebra, $P$ a projective $\Gamma$-module and $Q$ an injective $\Gamma$-module. Bearing in mind Proposition~\ref{p:uniqueness}, we denote (when they exist) the unique basic $P$-special tilting module by $\tiltmod{P}$ and the unique basic $Q$-special cotilting module by $\cotiltmod{Q}$. Their endomorphism algebras are denoted by
\[\tiltalg{P}=\End_\Gamma(\tiltmod{P})^{\op},\qquad\cotiltalg{Q}=\End_\Gamma(\cotiltmod{Q})^{\op}.\]
\end{dfn}

\begin{rem}
If $\add{P}=\add{P'}$, then there is a $P$-special tilting module if and only if there is a $P'$-special tilting module, and $\tiltmod{P}=\tiltmod{P'}$. The analogous statement also holds for $Q$-special cotilting modules, so without loss of generality we may always assume that $P$ and $Q$ in Definitions~\ref{dfn-tilting} and \ref{special-tilt-notation} are basic.
\end{rem}

\begin{exa}
\label{special-tilt-egs}
\begin{itemize}
\item[(1)] The first examples of tilting modules were APR-tilting modules \cite{APR}. Let $\Gamma $ be the path algebra of an acyclic quiver and $a$ a sink in the quiver, with at least one incoming arrow. Let $P=\bigoplus_{i\neq a} P(i)$. 
Then the unique basic $P$-special tilting module $T_P=P\oplus \tau^- S(a)$ is precisely the APR-tilting module. 

\item[(2)] The $1$-shifted tilting module and the $1$-coshifted cotilting module of an algebra $\Gamma$ of positive dominant dimension, studied by the authors in \cite{PS1} (see also \cites{HX,NRTZ}) are $\Pi$-special, where $\Pi$ additively generates the category of projective-injective $\Gamma$-modules.

\item[(3)] The characteristic tilting module $T$ of a right ultra-strongly quasihereditary algebra 
is special cotilting, by a theorem of Conde stated in the introduction to \cite{Conde}. In the notation of loc.\ cit., in which the indecomposable injective modules are indexed by pairs $(i,j)$ with $j\leq\ell_i$ for some $\ell_i\in\N$, the module $T$ is special cotilting for the injective $Q=\bigoplus_iQ_{i,\ell_i}$, the theorem showing that each indecomposable injective $Q_{i,j}$ with $j<\ell_i$ fits into an exact sequence
\[\begin{tikzcd}[column sep=20pt]
0\arrow{r}&T(i,j)\arrow{r}&Q_{i,\ell_i}\arrow{r}&Q_{i,j}\arrow{r}&0
\end{tikzcd}\]
for $T(i,j)$ a summand of $T$. Dually, the characteristic tilting module of a left ultra-strongly quasihereditary algebra is special tilting.
\end{itemize}
\end{exa}

\begin{rem}
\label{end-specialtilt}
Let $\Gamma $ be a finite-dimensional algebra, $Q\in\lmod{\Gamma}$ an injective module and $P\in\lmod{\Gamma}$ a projective module. Assume that there exists a $Q$-special cotilting module and a $P$-special tilting module. As usual, we denote the unique basic such modules by $C^Q$ and $T_P$, and their endomorphism algebras by $B^Q$ and $B_P$.

\begin{itemize}
\item[(1)]Since $Q\in\add{C^Q}$, the $B^Q$-module $\widetilde{P}= \Hom_{ \Gamma} (C^Q, Q)$ is projective. It then follows by applying $\Hom_\Gamma(C^Q,-)$ to the exact sequence in (C3) that $\kdual C^Q$ is the unique basic $\widetilde{P}$-special tilting $B^Q$-module.
 
\item[(2)]Dually, the $B_P$-module $\widetilde{Q}=\kdual  \Hom_{ \Gamma} (P , T_P )$ is injective, and $\kdual T_P$ is the unique basic $\widetilde{Q}$-special cotilting $B_P$-module. 

\end{itemize}
\end{rem}

The following lemma provides our most important source of special tilting and cotilting modules.

\begin{lem} \label{special-tilt}
Let $A$ be a finite-dimensional algebra and $E$ a finite-dimensional $A$-module. Write $\Gamma=\End_\Gamma(E)^{\op}$.

\begin{itemize}
\item[(1)] If $E$ is a cogenerator, then $P=\kdual E$ is a projective $\Gamma$-module and there is a unique basic $P$-special tilting module $T_P$ for $\Gamma $. Moreover, $\End_\Gamma(P)^{\op}\cong A$.

\item[(2)] If $E$ is a generator, then $Q=\kdual E$ is an injective $\Gamma $-module and there is a unique basic $Q$-special cotilting module $C^Q$ for $\Gamma $. Moreover, $\End_\Gamma(Q)^{\op}\cong A$.
\end{itemize}
\end{lem}
\begin{proof}
We prove only (1), statement (2) being dual. First, observe that
\[P=\kdual E=\Hom_A(E,\kdual A)\]
is projective, since $E$ is a cogenerator so $\kdual A\in\add{E}$. By Proposition~\ref{p:uniqueness}, it is enough to show the existence of a $P$-special tilting module. Let $f\colon E\to Q(E)$ be an injective envelope. Applying $\Hom_A(E,-)$ to this map and taking the cokernel yields a short exact sequence
\begin{equation}
\label{Pspecial-sequence}
\begin{tikzcd}[column sep=20pt]
0\arrow{r}&\Gamma\arrow{r}&P_0\arrow{r}&T_1\arrow{r}&0.
\end{tikzcd}
\end{equation}
Moreover,
\[P_0=\Hom_A(E,Q(E))\in\add\Hom_A(E,\kdual A)=\add{P}.\]
Let $T=P_0\oplus T_1$; we claim that $T$ is a $P$-special tilting module. Since $P_0$ is projective, property (T1) is provided by sequence \eqref{Pspecial-sequence}. Property (T3) is immediate from \eqref{Pspecial-sequence}. To show that (T2) holds, it is enough to prove that
\begin{itemize}
\item[(i)]$\Ext^1_\Gamma(T_1,P_0)=0$, and
\item[(ii)]$\Ext^1_\Gamma(T_1,T_1)=0$,
\end{itemize}
since $P_0$ is projective. For (i), apply $\Hom_\Gamma(-,P_0)$ to \eqref{Pspecial-sequence} to obtain an exact sequence
\[\begin{tikzcd}[column sep=20pt]
0\arrow{r}&\Hom_\Gamma(T_1,P_0)\arrow{r}&\Hom_\Gamma(P_0,P_0)\arrow{r}{g}&\Hom_\Gamma(\Gamma,P_0)\arrow{r}&\Ext^1_\Gamma(T_1,P_0)\arrow{r}&0.
\end{tikzcd}\]
We wish to show that $g$ is surjective. Consider the commutative diagram
\[\begin{tikzcd}[column sep=70pt]\
\Hom_\Gamma(P_0,P_0)\arrow{r}{g}\arrow{d}[anchor=north,sloped,above]{\sim}&\Hom_\Gamma(\Gamma,P_0)\arrow{d}[anchor=north,sloped,above]{\sim}\\
\Hom_A(Q(E),Q(E))\arrow{r}{\Hom_A(f,Q(E))}&\Hom_A(E,Q(E)).
\end{tikzcd}\]
in which the vertical maps are isomorphisms from Yoneda's lemma. Since $f\colon E\to Q(E)$ is an injective envelope of $E$, the map $\Hom_A(f,Q(E))$ is surjective, so $g$ is also surjective as required.

We now show that (ii) follows from (i). Applying various Hom-functors to sequence \eqref{Pspecial-sequence} yields the commutative diagram
\[\begin{tikzcd}[column sep=20pt]
0\arrow{r}&\Hom_\Gamma(T_1,P_0)\arrow{r}\arrow{d}&\Hom_\Gamma(P_0,P_0)\arrow{r}{g}\arrow{d}&\Hom_\Gamma(\Gamma,P_0)\arrow{r}\arrow{d}{q}&0\\
0\arrow{r}&\Hom_\Gamma(T_1,T_1)\arrow{r}&\Hom_\Gamma(P_0,T_1)\arrow{r}{p}&\Hom_\Gamma(\Gamma,T_1)\arrow{r}&\Ext^1_\Gamma(T_1,T_1)\arrow{r}&0.
\end{tikzcd}\]
By (i) we know that $g$ is surjective, and $q$ is also surjective since $\Gamma$ is projective. Thus $p$ is surjective, and (ii), and hence (T2), follows.

Since $E$ is a cogenerator, $\kdual{A}$ is a summand of $Q(E)$, and so $P\in\add{P_0}\subseteq\add{T}$. This, together with \eqref{Pspecial-sequence}, shows that $T$ is $P$-special.

The final statement follows by Yoneda's lemma, since
\[\End_\Gamma(P)^{\op}=\End_\Gamma(\Hom_A(E,\kdual A))^{\op}=\End_A(\kdual A)^{\op}=A,\]
using that $E$ is a cogenerator so $\kdual A\in\add{E}$.
\end{proof}

\begin{dfn}
Let $E\in\lmod{A}$ and $\Gamma=\End_A(E)^{\op}$.

\begin{itemize}
\item[(1)]If $E$ is a cogenerator, let $P$ be as in Lemma~\ref{special-tilt}(1). We call $B_P=\End_\Gamma(T_P)^\op$ the \emph{cogenerator-tilted algebra} of $E$, and the idempotent $e\in B_P$ given by projection onto $P$ is called \emph{special}.
\item[(2)]If $E$ is a generator, let $Q$ be as in Lemma~\ref{special-tilt}(2). We call $B^Q=\End_\Gamma(C^Q)^\op$ the \emph{generator-cotilted algebra} of $E$, and the idempotent $e\in B^Q$ given by projection onto $Q$ is called \emph{special}.
\end{itemize}
\end{dfn}

It follows from Lemma~\ref{special-tilt} that, for either $B=B_P$ or $B=B^Q$, the idempotent subalgebra $eBe$ defined by the special idempotent $e$ is isomorphic to $A$. Furthermore, it follows from Remark~\ref{end-specialtilt} that the canonical tilting $B^Q$-module $\kdual C^Q$ is the unique $B^Qe$-special tilting module and the canonical cotilting $B_P$-module $\kdual T_P$ is the unique $\kdual(eB_P)$-special cotilting module. We note that if $E$ is both a generator and a cogenerator then, in the terminology of \cite{PS1}, the cogenerator-tilted algebra of $E$ is the $1$-shifted algebra of $\Gamma$, and the generator-cotilted algebra of $E$ is the $1$-coshifted algebra of $\Gamma$.

\begin{exa}
Let $A$ be a finite-dimensional algebra, and fix a natural number $L$ such that $\rad^L(A)=0$. Let $E$ be a basic $A$-module such that
\[\add{E}=\add\left(\bigoplus_{i=1}^LA/\rad^i(A)\right),\]
and write $R_A=\End_{A}(E)^{\op}$. This construction is due to Auslander \cite{ART1}, and Dlab and Ringel showed that $R_A$ is quasihereditary \cite{DR-ADR}; hence $R_A$ is often called the ADR-algebra of $A$. Since $A/\rad^L(A)=A$, the module $E$ is a generator, and so $R_A$ admits a unique basic $\kdual E$-special cotilting module $C$ by Lemma~\ref{special-tilt}.

The algebra $R_A$ is even right ultra-strongly quasihereditary \cite{Conde}, so by Example~\ref{special-tilt-egs}(3), its characteristic tilting module is special cotilting for an injective module $Q=\bigoplus_iQ_{i,\ell_i}$. By \cite[Lem.~4.4]{Conde}, $Q=\kdual E$, and so in fact the characteristic tilting module is the special cotilting module $C$ from Lemma~\ref{special-tilt}. Thus in this case the generator-cotilted algebra of $E$ is, by definition, the Ringel dual of the quasihereditary algebra $R_A$ \cite[\S6]{Ri2}.
\end{exa}

Let $B$ be a finite-dimensional algebra, let $e\in B$ be an idempotent element and write $A=eBe$. We obtain from $e$ a diagram
\begin{equation}
\label{recollement}
\begin{tikzcd}[column sep=40pt]
\lmod{B/B eB}\arrow["i" description]{r}&\lmod{B}\arrow["e" description]{r}\arrow[shift right=3,swap]{l}{q}\arrow[shift left=3]{l}{p}&\lmod{A}\arrow[shift right=3,swap]{l}{\ell}\arrow[shift left=3]{l}{r}
\end{tikzcd}
\end{equation}
of six functors, defined by 
\[
\begin{aligned}
q & = B /B eB \otimes_{B } (-),  &&& \ell &= B e\otimes_A -,  \\
i & = B /B eB \otimes_{B /B eB } (-),   &&&  e &= \Hom_{B }(B e,-)=eB\otimes_B-,   \\
p & = \Hom_{B}(B/BeB , -),  &&& r &= \Hom_{A} (eB,- ).
\end{aligned}
\]
Such data is known as a recollement of abelian categories, and can be defined in abstract, but we will only consider recollements of module categories determined by idempotents as above (cf.\ \cite{PV}). For a $B$-module $M$, one obtains the same $A$-module $eM$ either by applying the functor $e$ in this diagram, or by multiplying on the left by the idempotent $e$, hence the abuse of notation. Since $\ell$ and $r$ are left and right adjoints of $e$ respectively, and $e\ell=er=\id$, there is a natural isomorphism
\[\Hom_B(\ell M,rM)\isoto\Hom_A(M,M),\]
functorial in $M$, and so determining a canonical map of functors $\ell\to r$. This map is equivalently described as the composition of the counit of the adjunction $\ell\dashv e$ with the unit of the adjunction $e\dashv r$. Taking its image yields  a seventh functor $c\colon \lmod{A} \to \lmod{B}$, called the intermediate extension \cite{Ku}. This functor will be particularly important in our geometric constructions, and so much of the algebraic part of the paper is devoted to studying it.

We recall from \cite{PS1} a description of the images and kernels of some of the functors appearing in the above recollement. This description uses the following notation.

\begin{dfn}
\label{d:gen-cogen}
Let $X\in A$-$\modu$ be a module. We write $\gen(X)$ for the full subcategory of $\lmod{A}$ consisting of modules admitting an epimorphism from an object of $\add{X}$, and $\gen_1(X)$ for the full subcategory of $\lmod{A}$ consisting of modules $Z$ fitting into an exact sequence
\[\begin{tikzcd}[column sep=20pt]
X_1\arrow{r}&X_0\arrow{r}&Z\arrow{r}&0
\end{tikzcd}\]
such that $X_i\in\add{X}$ and
\[\begin{tikzcd}[column sep=20pt]
\Hom_A(X,X_1)\arrow{r}&\Hom_A(X,X_0)\arrow{r}&\Hom_A(X,Z)\arrow{r}&0
\end{tikzcd}\]
is exact. We define $\cogen(X)$ and $\cogen^1(X)$ dually.
\end{dfn}

\begin{lem}[{\cite[Lem.~4.1]{PS1}}] \label{ImageAndKernel}
In the context of the idempotent recollement \eqref{recollement}, write $P=Be$ and $Q=\kdual(eB)$. Then
\begin{align*}
 \Ker q = \gen (P)&\supseteq\gen_1(P)=\Bild \ell, \\
\Ker p = \cogen (Q)&\supseteq\cogen^1(Q)=\Bild r.
\end{align*}
Moreover, the image of the intermediate extension $c=\Bild(\ell\to r)$ is given by
\[ 
\Bild c = \Ker p \cap \Ker q = \gen (P) \cap \cogen (Q).
\]
\end{lem}

The main conclusion of the algebraic part of the paper is the following theorem, which we will prove at the end of Section~\ref{homot-end-rings}.
\begin{thm}[{cf.\ \cite[Thm.~6.3, Thm.~6.5]{PS1}}]
\label{int-ext-calcs}
Let $A$ be a finite-dimensional algebra, and $E$ an $A$-module. We write $\Gamma = \End_A(E)^{\op}$. 
\begin{itemize}
\item[(1)]Assume $E$ is cogenerating, let $B$ be its cogenerator-tilted algebra with special idempotent $e$, and let $c$ be the intermediate extension functor corresponding to $e$. Write $P=\kdual E$, so that $B=\End_{\Gamma}(T_P)^{\op}$. Then
\[ c(E) \cong \kdual T_P\quad \text{ and }\quad c(\kdual A) \cong\kdual(eB).\]
In particular, $c(\kdual A)$ is injective, and $c(E)$ is the canonical cotilting module for the tilted algebra $B$, which is the unique $\kdual(eB)=c(\kdual A)$-special cotilting module by Remark~\ref{end-specialtilt}.

\item[(2)]Assume $E$ is generating, let $B$ be its generator-cotilted algebra with special idempotent $e$, and let $c$ be the intermediate extension functor corresponding to $e$. Write $Q=\kdual E$, so that $B=\End_{\Gamma}(C^Q)^{\op}$. Then
\[c(E)\cong\kdual C^Q\quad\text{ and }\quad c(A)\cong Be.\]
In particular, $c(A)$ is projective, and $c(E)$ is the canonical tilting module for the cotilted algebra $B$, which is the unique $Be=c(A)$-special tilting module by Remark~\ref{end-specialtilt}.
\end{itemize}
\end{thm}

For our applications, the most important consequence of Theorem~\ref{int-ext-calcs} is that, in either case, the fully faithful functor $c$ embeds $\add{E}$ into $\lmod{B}$ in such a way that there are no extensions between objects in the image (since $c(E)$ is either cotilting or tilting, and hence a rigid object). To prove this theorem, we will give a different description of the algebra $B$ in each part, as the endomorphism algebra of a bounded complex of $A$-modules in the homotopy category (cf.\ \cite[\S4.3]{PS1}).

\section{Endomorphism rings in the homotopy category}
\label{homot-end-rings}
\subsection{Homotopy categories and derived equivalence}

We begin by repeating some general principals from \cite[\S4.3]{PS1}. Let $A$ be a finite-dimensional algebra, $E\in\lmod{A}$, and $\Gamma=\End_A(E)^{\op}$. The bounded homotopy categories $\homot[\bound]{\proje{\Gamma}}$ and $\homot[\bound]{\injec{\Gamma}}$ of complexes of projective and injective $\Gamma$ modules respectively admit tautological functors to the unbounded derived category $\dercat{\Gamma}$, equivalences onto their images, which we treat as identifications. These subcategories may be characterised intrinsically as the full subcategories of $\dercat{\Gamma}$ on the compact and cocompact objects (in the context of additive categories) respectively. Extending the Yoneda equivalences
\begin{align*}
\Hom_A(E,-)&\colon\add{E}\isoto\proje{\Gamma},\\
\kdual\Hom_A(-,E)&\colon\add{E}\isoto\injec{\Gamma}
\end{align*}
to complexes, one sees that both of these subcategories of $\dercat{\Gamma}$ are equivalent to the subcategory $\thick(E)$ of the homotopy category $\homot[\bound]{A}$ of bounded complexes of arbitrary $A$-modules; by definition, $\thick(E)$ is the smallest full triangulated subcategory of $\homot[\bound]{A}$ closed under direct summands and containing (the stalk complex) $E$.

Now let $F\colon\mcT\isoto\dercat{\Gamma}$ be any equivalence of triangulated categories. It follows from the intrinsic description of $\homot[\bound]{\proje{\Gamma}}$ and $\homot[\bound]{\injec{\Gamma}}$ above that $F$ induces respective equivalences from the subcategories of compact and cocompact objects of $\mcT$ to $\homot[\bound]{\proje{\Gamma}}$ and $\homot[\bound]{\injec{\Gamma}}$ respectively, and thus realises $\thick(E)$ as a full subcategory of $\mcT$ (in two ways). This holds in particular when $\mcT=\dercat{B}$ for some algebra $B$, such as the endomorphism algebra of a tilting or cotilting $\Gamma$-module.

Whenever $B$ is derived equivalent to $\Gamma$, it follows from Rickard's Morita theory for derived categories \cite{RickMT} that the image in $\homot[\bound]{\proje{\Gamma}}$ of the stalk complex $B\in\homot[\bound]{\proje{B}}$ is a tilting complex with endomorphism algebra $B$, inducing the derived equivalence. The preimage of this tilting complex under the Yoneda equivalence is an object of $\thick(E)\subseteq\homot[\bound]{A}$, again with endomorphism algebra $B$. Similarly, the image of $\kdual B\in\homot[\bound]{\injec{B}}$ in $\homot[\bound]{\injec{\Gamma}}$ is a cotilting complex, and its preimage under the dual Yoneda equivalence is another object of $\thick(E)$ with endomorphism algebra $B$. Our conclusion is that when $\Gamma$ is the endomorphism algebra of an $A$-module $E$ (or more generally an object $E\in\homot[\bound]{A}$), any algebra $B$ derived equivalent to $\Gamma$ must also appear as an endomorphism algebra in $\thick(E)\subseteq\homot[\bound]{A}$. In general, $B$ need not be an endomorphism algebra in $\lmod{A}$.

In the context of Theorem~\ref{int-ext-calcs}, we may compute the relevant objects of $\thick(E)$ explicitly. This calculation generalises \cite[Prop.~5.5]{CBSa} for the case that $A$ is representation-finite and $\add{E}=\lmod{A}$, a connection that we will expand on in the next subsection.

\begin{pro}
\label{identifyB}
Let $A$ be a finite-dimensional basic algebra and let $E\in\lmod{A}$ be a basic module. Write $\Gamma=\End_A(E)^{\op}$.
\begin{itemize}
\item[(1)]Assume $E$ is a cogenerator, and write $P=\kdual E$ (see Lemma~\ref{special-tilt}(1)). Then
\[B_P\cong\End_{\homot[\bound]{A}}\Big(E\xrightarrow{\left(\begin{smallmatrix}f\\0\end{smallmatrix}\right)} Q(E)\oplus\kdual{A}\Big)^{\op},\]
where $f\colon E\to Q(E)$ is a minimal injective envelope. The special idempotent $e\in B_P$ given by projection onto $P$ corresponds under this isomorphism to projection onto the summand $0\to\kdual A$.
\item[(2)]Assume $E$ is a generator, and write $Q=\kdual E$ (see Lemma~\ref{special-tilt}(2)). Then
\[B^Q\cong\End_{\homot[\bound]{A}}\Big(P(E)\oplus A\xrightarrow{(\begin{smallmatrix}g&0\end{smallmatrix})}E\Big)^{\op},\]
where $g\colon P(E)\to E$ is a minimal projective cover. The special idempotent $e\in B^Q$ given by projection onto $Q$ corresponds under this isomorphism to projection onto the summand $A\to0$.
\end{itemize}
\end{pro}
\begin{proof}
As usual, we prove only (1), since (2) is dual. By definition $B_P$ is the endomorphism algebra of the unique basic $P$-special tilting $\Gamma$-module $T_P$, so that the image of $B_P$ in $\homot[\bound]{\proj{\Gamma}}$ is given by a projective resolution of $T_P$. By the proof of Lemma~\ref{special-tilt}, we have $T_P=T_1\oplus\kdual E$ and an exact sequence
\[\begin{tikzcd}[column sep=20pt]
0\arrow{r}&\Gamma\arrow{r}&P_0\arrow{r}&T_1\arrow{r}&0
\end{tikzcd}\]
in which the map $\Gamma\to P_0$ is the image under $\Hom_A(E,-)$ of a minimal injective envelope $f\colon E\to Q(E)$. Thus a projective resolution of $T_P$ is given by the direct sum of the map $\Gamma\to P_0$ above with the zero map $0\to\kdual E$, treated as a $2$-term complex. Taking the preimage of this complex under the Yoneda equivalence $\Hom_A(E,-)$ yields
\[(E\xrightarrow{\left(\begin{smallmatrix}f\\0\end{smallmatrix}\right)} Q(E)\oplus\kdual{A})\in\mcK^b(A),\]
and the required isomorphism follows. Since $0\to\kdual E$, corresponding under Yoneda to $0\to\kdual A$, is the part of the projective resolution of $T_P$ contributed by the summand $P$, we have the claimed relationship between idempotents.
\end{proof}

\begin{rem}
We did not specify degrees in the complexes on the right-hand side of the isomorphisms of Proposition~\ref{identifyB}, since such a choice plays no role in the statement. When such concreteness is required, we take the term $E$ to be in degree $0$ in each case.
The assumptions on basicness of $A$ and $E$ and minimality of the relevant projective cover and injective envelope are necessary since $B^I$ and $B_P$ are basic algebras by construction. However, one can remove all these assumptions from the statement at the cost of replacing the isomorphisms by Morita equivalences.
\end{rem}

\subsection{The category of injective envelopes}
\label{injenvelopes}

In the case that $A$ is representation-finite and $E$ is an additive generator of $\lmod{A}$, a useful description of the recollement arising from the special idempotent of the generator-cotilted algebra of $E$ is given in \cite[\S3]{CBSa} via the category $\mcH$ of projective quotients. We will briefly recall this construction, and generalise some of the results to our setting, in Section~\ref{projquots} below. However, since in the geometric applications to follow we have opted to use cogenerator-tilted algebras instead, we give more details for this case, in which we use instead the dual notion of a category $\check{\mcH}$ of injective envelopes.

We define $\check{\mcH}$ via the following construction, dual to that in \cite[\S3]{CBSa}. Let $\mcQ$ be the category with objects the monomorphisms $X\to Q$ of $A$-modules for which $Q$ is injective, and morphisms given by commuting squares. Then $\check{\mcH}$ is obtained from $\mcQ$ as the quotient by those morphisms factoring through an object of the form $\id_Q\colon Q\to Q$ for $Q$ injective.

We may view this category in a different way; first we identify $\mcQ$ with a full subcategory of the category $\mcC^b(A)$ of bounded chain complexes of $A$-modules, by interpreting the objects $X\to Q$ of $\mcQ$ as complexes with $X$ in degree $0$. It is then straightforward to check that a map between such complexes factors through a complex of the form $\id_Q\colon Q\to Q$ if and only if it is null-homotopic, so that $\check{\mcH}$ is identified with the full subcategory of $\mcK^b(A)$ on the same objects as $\mcQ$.

Now consider case (1) from Proposition~\ref{identifyB}, so that $E\in\lmod{A}$ is a cogenerator, and write
\[Q^E=\Big(E\xrightarrow{\left(\begin{smallmatrix}f\\0\end{smallmatrix}\right)} Q(E)\oplus\kdual{A}\Big),\]
so that the cogenerator-tilted algebra $B$ of $E$ satisfies
\[B:=B_P\cong\End_{\mcK^b(A)}(Q^E)^{\op}\]
for $P=\kdual E$. Since $f$ is an injective envelope, $Q^E\in\check{\mcH}$ (under our convention that the term $E$ is in degree $0$). We write $\check{\mcH}_E=\add{Q^E}$ for the additive closure of $Q^E$ in $\check{\mcH}$, or equivalently in $\mcK^b(A)$. By Proposition~\ref{identifyB}, a $B$-module is the same as an $\check{\mcH}_E$-module, and the functor $e\colon\lmod{B}\to\lmod{A}$ corresponds to restricting functors on $\check{\mcH}_E$ to the full subcategory $\check{\mcH}_0$ on objects of the form $0\to Q$.

We now have a collection of restriction functors
\begin{align*}
e&\colon\lmod{\check{\mcH}_E}\to\lmod{\check{\mcH}_0},&
e_E&\colon\lmod{\check{\mcH}}\to\lmod{\check{\mcH}_E},&
\widehat{e}&\colon\lmod{\check{\mcH}}\to\lmod{\check{\mcH}_0}
\end{align*}
with $\widehat{e}=e_Ee$. Taking left and right adjoints, we obtain a diagram
\[\begin{tikzcd}[column sep=60pt,row sep=60pt]
\lmod{\check{\mcH}}\arrow["\widehat{e}" description]{r}\arrow["e_E" description]{d}&\lmod{\check{\mcH}_0}\simeq\lmod{A}\arrow[equal, shift right=8]{d}\arrow[shift right=3,swap]{l}{\widehat{\ell}}\arrow[shift left=3]{l}{\widehat{r}}
\\
\lmod{\check{\mcH}_E}\arrow["e" description]{r}\arrow[shift left=3]{u}{\ell_E}\arrow[shift right=3,swap]{u}{r_E}&\lmod{\check{\mcH}_0}\simeq\lmod{A}\arrow[shift right=3,swap]{l}{\ell}\arrow[shift left=3]{l}{r}
\end{tikzcd}\]
and intermediate extension functors
\begin{align*}
c&=\Bild(\ell\to r),&c_E&=\Bild(\ell_E\to r_E),&\widehat{c}&=\Bild(\widehat{\ell}\to\widehat{r}).
\end{align*}
We are now able to give an explicit description and several properties of the functor $c$, by exploiting similar calculations for $\widehat{c}$ in \cite{CBSa}.

\begin{lem}
\label{c-description}
We have $c=e_E\widehat{c}$. Moreover,
\[c(M)(X\to Q)=\ker(\kdual\Hom_A(M,Q)\to\kdual\Hom_A(M,X))\]
for $(X\to Q)\in\check{\mcH}_E$.
\end{lem}
\begin{proof}
Since $\widehat{e}=e_Ee$, it follows by uniqueness of adjoints that $\widehat{\ell}=\ell_E\ell$ and $\widehat{r}=r_Er$. Post-composing with $e_E$, we see that $e_E\widehat{\ell}=\ell$ and $r_E\widehat{r}=r$.
By definition of $\widehat{c}$, there is an epimorphism $\pi\colon\widehat{\ell}\to\widehat{c}$ and a monomorphism $\iota\colon\widehat{c}\to\widehat{r}$, with $\iota\pi$ equal to the canonical map $\widehat{\ell}\to\widehat{r}$. Precomposing all functors with the exact functor $e_E$, we obtain the canonical map $\ell\to r$. The induced map $\ell\to e_E\widehat{c}$ is again an epimorphism, and $e_E\widehat{c}\to r$ is again a monomorphism, and so $c=e_E\widehat{c}$, as claimed. Dualising \cite[Lem.~4.2]{CBSa} (cf.\ \cite[Thm.~4.12]{PS1}) gives
\[\widehat{c}(M)(X\to Q)=\ker(\kdual\Hom_A(M,Q)\to\kdual\Hom_A(M,X))\]
for all $(X\to Q)\in\mcH$. Since $c(M)=e_E\widehat{c}(M)$ is obtained by restriction of functors, it has the same formula when evaluated on $(X\to Q)\in\mcH_E$.
\end{proof}

\begin{thm} \label{old-proj+inj-res}
Let $E\in\lmod{A}$ be a cogenerator, and let $c\colon\lmod{A}\to\lmod{\check{\mcH}_E}$ be the intermediate extension functor. If $M\in\add{E}$ and $M\to Q(M)$ is an injective envelope, then $c(M)$ has an injective resolution
\[\begin{tikzcd}[column sep=20pt]
0\arrow{r}&c(M)\arrow{r}&\kdual\Hom_{\check{\mcH}_E}(0\to Q(M),-)\arrow{r}&\kdual\Hom_{\check{\mcH}_E}(M\to Q(M),-)\arrow{r}&0.
\end{tikzcd}\]
In particular, $\idim_{\lmod{\check{\mcH}_E}}c(M)\leq 1$. Furthermore, $\Ext^1_{\lmod{\check{\mcH}_E}}(c(N),c(M))=0$ for any $N\in\lmod{A}$.
\end{thm}

\begin{proof} 
By dualising \cite[Lem.~4.3]{CBSa}, we see that the the sequence
\[\begin{tikzcd}[column sep=20pt]
0\arrow{r}&\widehat{c}(M)\arrow{r}&\kdual\Hom_{\check{\mcH}}(0\to Q(M),-)\arrow{r}&\kdual\Hom_{\check{\mcH}}(M\to Q(M),-)\arrow{r}&0.
\end{tikzcd}\]
is an injective resolution of $\widehat{c}(M)$. When $M\in\add{E}$, applying the exact functor $e_E$ gives the desired injective resolution of $c(M)$, since $e_E\widehat{c}=c$, and by definition
\[ 
e_E(\kdual\Hom_{\check{\mcH}}(X\to Q,-)) = \kdual\Hom_{\check{\mcH}_{E}} (X\to Q,-) 
 \]
 whenever $(X\to Q)\in\check{\mcH}_E$. Applying $\Hom_{\lmod{\check{\mcH}_E}}(c(N),-)$ to this injective resolution, and using that $c$ is fully faithful, we obtain an exact sequence
\[\begin{tikzcd}[column sep=10pt]
0\arrow{r}&\Hom_A(N,M)\arrow{r}&\kdual c(N)(0\to Q(M))\arrow{r}&\kdual c(N)(M\to Q(M))\arrow{r}&\Ext^1_{\lmod{\mcH_E}}(c(N),c(M))\arrow{r}&0
\end{tikzcd}\]
Using the calculation of $c(N)$ in Lemma~\ref{c-description}, we see that the middle map is surjective, being dual to the tautological injection
\[ 
\ker\big(\kdual\Hom_A(N,Q(M))\to\kdual\Hom_A(N,M)\big)\to\kdual\Hom_A(N,Q(M)),\]
and the statement follows.
\end{proof}

By Proposition~\ref{identifyB}, there is an equivalence of categories $\lmod{\check{\mcH}_E}\to\lmod{B}$, given by evaluation on the additive generator $Q^E$ of $\check{\mcH}_E$, which we will often treat as an identification. 
Using this, we may now prove Theorem~\ref{int-ext-calcs}(1).

\begin{proof}[Proof of Theorem~\ref{int-ext-calcs}(1)]
We first show that $c(\kdual A)\cong\kdual(eB)$ for $e$ the special idempotent. The identity map $\kdual A\to\kdual A$ is an injective envelope, and is isomorphic to the zero object in $\check{\mcH}_E$. Thus, after evaluating on $Q^E$ to identify $\lmod{\check{\mcH}_E}$ with $\lmod{B}$, the injective resolution of $c(\kdual A)$ from Theorem~\ref{old-proj+inj-res} provides an isomorphism
\[c(\kdual A)\cong\kdual\Hom_{\check{\mcH}_E}(0\to\kdual A,Q^E)\cong\kdual(eB)\]
by Proposition~\ref{identifyB}.

Recall that $B=\End_\Gamma(T_P)^{\op}$, where $\Gamma=\End_A(E)^{\op}$ and $P$ is the projective $\Gamma$-module $\kdual E$. By Remark~\ref{end-specialtilt}, the canonical cotilting $B$-module $\kdual T_P$ is the unique basic $\kdual(eB)$-special cotilting module. Thus to show that $c(E)=\kdual T_P$, it is enough to show that it is such a module.

Since $E$ is basic, so is $c(E)$. Moreover, $c(E)$ satisfies (C1) and (C2) by Theorem~\ref{old-proj+inj-res}. Interpreting the terms as $B$-modules by evaluation on $Q^E$, and using Proposition~\ref{identifyB}, the injective resolution of $c(E)$ from Theorem~\ref{old-proj+inj-res} becomes
\[\begin{tikzcd}[column sep=20pt]
0\arrow{r}&c(E)\arrow{r}&\kdual\Hom_{\check{\mcH}_E}(0\to Q(E),Q^E)\arrow{r}&\kdual((1-e)B)\arrow{r}&0
\end{tikzcd}\]
Since $(0\to Q(E))\in\add(0\to\kdual A)$, the middle term lies in $\add{\kdual(eB)}$ by Proposition~\ref{identifyB}. Adding the identity map $\kdual(eB)\to\kdual(eB)$ to the right-hand end of the sequence yields
\[\begin{tikzcd}[column sep=20pt]
0\arrow{r}&c(E)\arrow{r}&\widetilde{Q}\arrow{r}&\kdual B\arrow{r}&0
\end{tikzcd}\]
with $\widetilde{Q}\in\add{\kdual(eB)}$. This sequence shows that $c(E)$ satisfies (C3). Since $E$ is a cogenerator, $\kdual(eB)=c(\kdual A)\in\add{c(E)}$, which together with the previous sequence shows that $c(E)$ is $\kdual(eB)$-special, completing the proof.
\end{proof}

\subsection{The category of projective quotients}
\label{projquots}

All of the results of the previous section have dual analogues, leading to a proof of Theorem~\ref{int-ext-calcs}(2). We merely state the dual results, which correspond more closely to those of \cite{CBSa}. Let $\mcH$ be the category with objects given by surjective maps $P\to X$ of $A$-modules for which $P$ is projective, and morphisms by commuting squares modulo maps factoring through an object of the form $\id_P\colon P\to P$. Just as for $\check{\mcH}$, we may view $\mcH$ as a full subcategory of $\mcK^b(A)$ on the maps $P\to X$, thought of as $2$-term complexes with $X$ in degree $0$.

Assume $E\in\lmod{A}$ is a generator, and write $\Gamma=\End_A(E)^{\op}$. Writing $B$ for the generator-cotilted algebra of $E$, and
\[\mcH_E=\add\Big(P(E)\oplus A\xrightarrow{(\begin{smallmatrix}g&0\end{smallmatrix})}E\Big),\]
where $g\colon P(E)\to E$ is a minimal projective cover, Proposition~\ref{identifyB} shows that $\lmod{\mcH_E}$ and $\lmod{B}$ are equivalent categories via evaluating functors in $\lmod{\mcH_E}$ on the above additive generator. This identifies the restriction functor $e\colon\lmod{B}\to\lmod{A}$, induced from the special idempotent, with the restriction of functors in $\lmod{\mcH_E}$ to $\add(A\to0)$.

\begin{lem}[Dual to Lemma~\ref{c-description}] \label{c-description*}
The intermediate extension functor $c\colon\lmod{A}\to\lmod{\mcH_E}$ is given by
\[c(M)(P\to X)=\coKer(\Hom_A(X,M)\to\Hom_A(P,M))\]
for $(P\to X)\in\mcH_E$.
\end{lem}

\begin{thm}[Dual to Theorem~\ref{old-proj+inj-res}]
\label{dual-projres}
Let $E\in\lmod{A}$ be a generator, and let $c\colon\lmod{A}\to\lmod{\mcH_E}$ be the intermediate extension functor.
If $M\in\add{E}$ and $p\colon P(M)\to M$ is a projective cover, then $c(M)$ has a projective resolution 
\[\begin{tikzcd}[column sep=20pt]
0\arrow{r}&\Hom_{\mcH_E}(-,P(M)\xrightarrow{p}{M})\arrow{r}&\Hom_{\mcH_E}(-,P(M)\to 0)\arrow{r}&c(M)\arrow{r}&0.
\end{tikzcd}\]
In particular, $\pdim_{\lmod{\check{\mcH}_E}} c(M)\leq 1$.
Furthermore, $\Ext^1_{\lmod{\mcH_E}}(c(M), c(N))=0$ for any $N\in\lmod{A}$.
\end{thm}

Theorem~\ref{int-ext-calcs}(2) then follows from Proposition~\ref{identifyB} and Theorem~\ref{dual-projres} via a dual argument to that given above for part (1).

\section{Rank varieties and orbit closures}
\label{orbit-closures}

\subsection{Rank varieties as affine quotient varieties}
\label{rankvars}

We now turn to the geometric part of the paper. In this section, we assume that $\K$ is an algebraically closed field of characteristic zero, so that we may freely use results from \cite[\S6]{CBSa}, and let $A=\K Q/I$ be a finite-dimensional algebra presented by a finite quiver $Q$ and an admissible ideal $I$. For finite-dimensional $A$-modules $X$ and $Y$, write
\[\left[ X,Y\right] =\dim\Hom_A(X,Y).\]
Write $Q_{0}=\{1, \ldots , n\} $ and $e_i$ for the primitive idempotent corresponding to $i\in Q_0$. We write $\Dim{M}= (\dim e_iM)_{1 \leq i\leq n}\in\Z^{n}$ for the dimension vector of an $A$-module $M$. Given $v,w\in\Z^{n}$, we write $v\leq w$ if this inequality holds componentwise.

For $d\in\Z^n_{\geq0}$, we denote by
\[\rep{Q}{d}= \prod_{(i\to j)\in Q_1 } \Hom_{\K }(\K^{d_i}, \K^{d_j} )\]
the representation space of $Q$, each point of which defines a $\K Q$-module with dimension vector $d$ in the usual way. The representation space of $A$ is then the closed subvariety 
\[\rep{A}{d} = \{ M \in \rep{Q}{d} \mid IM=0\}\]
of collections of maps satisfying the relations in $I$. This space carries a natural action of the algebraic group $\Gl_d:=\prod_{i=1}^n\Gl_{d_i}$, with orbits corresponding to isomorphism classes of $d$-dimensional $A$-modules.

\begin{dfn} Let $E=\bigoplus_{i=1}^t E_i$, where each $E_i\in\lmod{A}$ is indecomposable, and let $m=(m_1,\ldots , m_t) \in \Z_{\geq 0}^t$. We define 
\[ 
\begin{aligned}
\drankvar{m}{E}&:= \{ N\in \rep{A}{d} \mid \left[ N, E_i\right] \geq m_i, \; 1\leq i\leq t \},\\
\rankvar{m}{E}&:= \{ N\in \rep{A}{d} \mid \left[E_i, N\right] \geq m_i, \; 1\leq i\leq t \}.
\end{aligned}
\]
Since the maps $\rep{A}{d}\to \Z_{\geq 0}$ defined by $N\mapsto \left[N, X\right ] $ and $N\mapsto \left[X,N\right ] $ are upper-semicontinuous for every module $X$, the subsets $\drankvar{m}{E}$ and $\rankvar{m}{E}$ are Zariski-closed in $\rep{A}{d}$. We will refer to them as \emph{rank varieties}. For any fixed module $M\in\rep{A}{d}$, we write $\drankvar{M}{E}:= \drankvar{m}{E}$ where $m_i =\left[ M, E_i \right]$ and $\rankvar{M}{E}:=\rankvar{m}{E}$ where $m_i:=\left[E_i,M\right]$. For such an $M$, we also study  
\[ 
\begin{aligned}
\drankvar{M}{} &= \{ N\in \rep{A}{d} \mid \left[ N, U\right] \geq \left[ M, U\right] \text{ for all }U\in\lmod{A} \},\\
\rankvar{M}{} &= \{ N\in \rep{A}{d} \mid \left[ U, N\right] \geq \left[ U,M\right] \text{ for all }U\in\lmod{A} \},
\end{aligned}
\]
which are also closed sets in $\rep{A}{d}$ \cite[Prop.~1(a)]{Bo6}. Note that $\drankvar{M}{}$ is the intersection of the sets $\drankvar{M}{U}$ as $U$ runs over all indecomposable $A$-modules, and similarly for $\rankvar{M}{}$. 
\end{dfn}

It follows from a result of Auslander and Reiten \cite[Thm.~1.4]{AR-compfactors} that for $M,N\in\rep{A}{d}$ and $U$ indecomposable, one has $[U,N]\geq[U,M]$ if and only if $[N,\tau U]\geq[M,\tau U]$. Since $[P,M]$ and $[M,Q]$ are determined by $\Dim{M}$ when $P$ is projective and $Q$ is injective, and $\tau$ gives a bijection between non-projective indecomposables and non-injective indecomposables, we have $\drankvar{M}{\tau E}=\rankvar{M}{E}$. It follows in the same way that $\drankvar{M}{}=\rankvar{M}{}$; we call this space \emph{the rank variety of $M$}, and typically opt for the notation $\drankvar{M}{}$. This variety has been studied by Bongartz \cite{Bo} and (as a scheme) by Riedtmann--Zwara \cite{RZ}, among others.

\begin{rem}
\label{bigger-E}
Setting $m'_i=m_i$ for $1\leq i\leq t$ and $m_{t+1}'=0$, we have $\drankvar{m}{E}=\drankvar{m'}{E\oplus E_{t+1}}$ and $\rankvar{m}{E}=\rankvar{m'}{E\oplus E_{t+1}}$ for any indecomposable $E_{t+1}$. If $Q$ is injective and $P$ is projective then, since $[M,Q]$ and $[P,M]$ depend only on $\Dim{M}$, it follows that $\drankvar{M}{E}=\drankvar{M}{E\oplus Q}$ and $\rankvar{M}{E}=\rankvar{M}{E\oplus P}$. Thus, when discussing $\drankvar{M}{E}$ or $\drankvar{m}{E}$ we may always assume that $E$ is a cogenerator, and when discussing $\rankvar{M}{E}$ or $\rankvar{m}{E}$ we may always assume that $E$ is a generator, without any loss of generality.

Moreover, if $E$ is such that $\drankvar{M}{E}=\drankvar{M}{}$, we have $\drankvar{M}{E\oplus X}=\drankvar{M}{}$ for any $X\in\lmod{A}$. Thus in this case we are able to assume without loss of generality that the module $E$ in question has any property that may be acquired by taking the direct sum with another module, such as being generating, cogenerating, or satisfying $\gldim\End_A(E)^{\op}<\infty$ \cite[Thm.~1.1]{IyFRD}. The analogous statement holds when $\rankvar{M}{E}=\drankvar{M}{}$.
\end{rem}

By Hilbert's basis theorem, for any $M$ there exist modules $E$ and $E'$ such that $\drankvar{M}{}=\drankvar{M}{E}=\rankvar{M}{E'}$, although $E$ and $E'$ are neither explicitly nor uniquely determined. As a result, it is rarely clear how to find such modules, an obvious exception being when $A$ is representation-finite, in which case both can be taken to be additive generators of $\lmod{A}$.

\begin{rem}
\label{r:rankvars-orbits}
Under certain conditions on $A$, such as if $A$ is representation-finite \cite{Zw5} or tame concealed \cite{Bo6}, the rank variety $\drankvar{M}{}$ coincides (as a set, and as a scheme if one uses the reduced scheme structure) with the closure $\clorbit{M}$ of the $\Gl_d$-orbit of $M$ in $\rep{A}{d}$. Precisely, this happens if and only if the Hom-order and degeneration order coincide \cite[Prop.~1]{Bo6}. In general, the same proposition shows that $\clorbit{M}$ is an irreducible component of (the reduced scheme) $\drankvar{M}{}$. Since there are set theoretic inclusions $\clorbit{M}\subseteq\drankvar{M}{}\subseteq\drankvar{M}{E}$ for any $E\in\lmod{A}$, we have $\clorbit{M}=\drankvar{M}{E}$ for some $E$ if and only if $\clorbit{M}=\drankvar{M}{}$.
\end{rem}

A result of Zwara \cite[Thm.~1.2(4)]{Zw2} allows us to produce modules $E$ and $E'$ with $\drankvar{M}{}=\drankvar{M}{E}=\rankvar{M}{E'}$ more explicitly under certain finiteness conditions on the module $M$, which we now introduce. We use the categories $\gen(M)$ and $\cogen(M)$ from Definition~\ref{d:gen-cogen}.

\begin{dfn}
\label{gen-finite}
Let $M$ be a finite-dimensional $A$-module. We say $M$ is \emph{gen-finite} if there is a finite-dimensional $A$-module $E$ such that $\gen (M) = \add{E}$. Dually, we say $M$ is \emph{cogen-finite} if there is a finite-dimensional $A$-module $E$ such that $\cogen (M) = \add{E}$. 
\end{dfn}

The regular module $A$ is gen-finite (or equivalently $\kdual A$ is cogen-finite) if and only if $A$ is representation-finite. Thus we see gen-finiteness and cogen-finiteness as module-theoretic generalisations of the notion of representation-finiteness for algebras. For gen-finite and cogen-finite modules, rank varieties coincide with orbit closures, as follows.

\begin{thm}[{cf.\ \cite[Thm.~1.2(4)]{Zw2}}]
\label{rankvar=clorbit}
Let $M\in \rep{A}{d}$. If $\gen(M)=\add{E}$, then $\clorbit{M}=\drankvar{M}{E}$, and hence both are equal to $\drankvar{M}{}$ by Remark~\ref{r:rankvars-orbits}. Similarly, if $\cogen(M)=\add{E}$, then $\clorbit{M}=\rankvar{M}{E}=\rankvar{M}{}=\drankvar{M}{}$.
\end{thm}

As in \cite{CBSa}, our first step in constructing a desingularisation of $\drankvar{M}{}$ for an $A$-module $M$ is to realise it as an affine quotient of some variety of representations for another algebra $B$. In fact, we may do this for any of the varieties $\drankvar{m}{E}$ or $\rankvar{m}{E}$.

We begin with $\drankvar{m}{E}$. Assuming without loss of generality that $E$ is a basic cogenerator (see Remark~\ref{bigger-E}), we may decompose $E=\bigoplus_{j=1}^tE_j$, with $E_j=\kdual(e_jA)$ indecomposable injective for $1\leq j\leq n$, and indecomposable non-injective otherwise. 

Let $B$ be the cogenerator-tilted algebra of $E$, and let $e$ be its special idempotent. Since $A\cong eBe$, we may choose a complete set of primitive orthogonal idempotents of $B$ extending that of $A$, and thus write the dimension vector of a $B$-module $X$ as $(d,s)\in\Z^n\times\Z^{t-n}$, where $d=\Dim{eX}$ and $s=\Dim (1-e)X$. We number the components of $d$ from $1$ to $n$, and those of $s$ from $n+1$ to $t$; this is compatible with our numbering of the indecomposable summands of $E$. By restricting the usual action of $\Gl_{(d,s)}=\Gl_d\times\Gl_s$ on $\rep{B}{d,s}$, we get an action of $\Gl_s$ on this representation variety, and may consider the categorical quotient $\rep{B}{d,s}\gitquot\Gl_s$, the affine variety with coordinate ring given by the ring of invariants $\K[\rep{B}{d,s}]^{\Gl_s}$. Now the restriction functor $e$ provides a map
\[e\colon\rep{B}{d,s}\to\rep{A}{d},\]
and by \cite[Lem.~6.3]{CBSa} an induced isomorphism of varieties 
\[
\rep{B}{d,s}\gitquot \Gl_s \isoto \Bild e.
\]
Furthermore, $\Bild e=\{ N\in \rep{A}{d}\mid \Dim c(N) \leq (d,s)\}$ is a closed subset of $\rep{A}{d}$ \cite[Lem.~7.2]{CBSa}; here $c$ is the intermediate extension functor associated to the idempotent $e$. For any injective $A$-module $Q$ and any dimension vector $d$, let
\[\left[d,Q\right]:=\left[N, Q\right] \]
where $N\in\rep{A}{d}$ is arbitrary, noting that $\left[N,Q\right]$ depends only on $\Dim{N}=d$ by injectivity of $Q$. Since this quantity also only depends on $Q$ up to isomorphism, for any $X\in\lmod{A}$ we get a well-defined integer $\left[d,Q(X)\right]$, where $X\to Q(X)$ is a minimal injective envelope. 
We may now state the main result of this subsection, of which Theorem~\ref{affine-quot-intro} is a special case.

\begin{thm}
\label{affine-quot}
Let $E$ be a cogenerating $A$-module, with indecomposable summands labelled as in the preceding paragraph, and $B$ its cogenerator-tilted algebra. Let $d\in\Z_{\geq0}^n$ be a dimension vector for $A$, and let $m\in \Z_{\geq 0}^t$.
\begin{itemize}
\item[(1)] If $\drankvar{m}{E}\ne\varnothing$, then $[d,Q(E_j)]\geq m_j$ for all $j$.
\item[(2)] In this case, we may extend $d$ to a dimension vector $(d,s)\in\Z_{\geq0}^t$ for $B$ by defining $s_j:= [d,Q(E_j)]-m_j$ for $n+1\leq j\leq t$, and the special idempotent $e$ of $B$ induces an isomorphism
\[
 \rep{B}{d,s}\gitquot \Gl_s \isoto \drankvar{m}{E}.
\]
\end{itemize}
\end{thm}

\begin{proof}
Assume $N\in\drankvar{m}{E}\ne\varnothing$. For $1\leq j\leq n$, we have $E_j=\kdual(e_jA)$, and so
\[[d,Q(E_j)]=[d,E_j]=d_j=[N,E_j]\geq m_j.\]
On the other hand, if $n+1\leq j\leq t$ then, by Lemma~\ref{c-description}, we have $\Dim{c(N)}=(d,s')$, where
\begin{equation}
\label{dim-c-calc}
0\leq s'_j=[N,Q(E_j)]-[N,E_j]=[d,Q(E_j)]-[N,E_j],
\end{equation}
and so $[d,Q(E_j)]\geq[N,E_j]\geq m_j$.

Now, as discussed before the statement of the Proposition, it follows from \cite[Lem.~6.3, Lem.~7.2]{CBSa} that the special idempotent $e$ induces an isomorphism
\[\rep{B}{d,s}\gitquot\Gl_s \isoto \Bild e = \{ N\in \rep{A}{d}\mid \Dim c(N) \leq (d,s)\},\]
so it is enough to prove that $\drankvar{m}{E}$ coincides with the codomain. We have $[N,E_j]=d_j\geq m_j$ for any $N\in\rep{A}{d}$ if $1\leq j\leq n$, and so comparing $\Dim{c(N)}$, as calculated in \eqref{dim-c-calc}, to the definition of $s_j$ for $n+1\leq j\leq t$ we see that $\Dim{c(N)}\leq(d,s)$ if and only if $[N,E_j]\geq m_j$ for such $j$.
\end{proof}

\begin{rem}
In conjunction with Lemma~\ref{c-description}, the calculation in the proof of Theorem~\ref{affine-quot} shows that if $\drankvar{m}{E}=\drankvar{M}{E}$ for some $A$-module $M$, i.e.\ if $m_i=[M,E_i]$, then the dimension vector $(d,s)$ for $B$ appearing in part (2) of the theorem is precisely the dimension vector of $c(M)$.
\end{rem}

We now state the dual result for $\rankvar{m}{E}$, which may be proved similarly. In this case we may assume $E$ is a basic generator, and decompose $E=\bigoplus_{j=1}^tE_j$ so that $E_j=Ae_j$ is indecomposable projective for $1\leq j\leq n$, and $E_j$ is indecomposable non-injective otherwise. Let $B$ be the generator-cotilted algebra of $E$, and $e$ its special idempotent, so that again we have $A\cong eBe$. Dual to the earlier statement for injective envelopes, we get a well defined integer $[P(X),d]$ for any $X\in\lmod{A}$ and any dimension vector $d$, by taking $P(X)\to X$ to be a minimal projective cover. The dual of Theorem~\ref{affine-quot} is then the following.
\begin{thm}
\label{affine-quot*}
Let $E$ be a generating $A$-module, with indecomposable summands labelled as in the preceding paragraph, and $B$ its generator-cotilted algebra. Let $d\in\Z_{\geq0}^n$ be a dimension vector for $A$, and let $m\in \Z_{\geq 0}^t$.
\begin{itemize}
\item[(1)] If $\rankvar{m}{E}\ne\varnothing$, then $[P(E_j),d]\geq m_j$ for all $j$.
\item[(2)] In this case, we may extend $d$ to a dimension vector $(d,s)\in\Z_{\geq0}^t$ for $B$ by defining $s_j:= [P(E_j),d]-m_j$ for $n+1\leq j\leq t$, and the special idempotent $e$ of $B$ induces an isomorphism
\[
 \rep{B}{d,s}\gitquot \Gl_s \isoto \rankvar{m}{E}.
\]
\end{itemize}
\end{thm}

As remarked earlier, Hilbert's basis theorem allows us to apply Theorems~\ref{affine-quot} and \ref{affine-quot*} to the rank variety $\drankvar{M}{}$ of a module $M$, by expressing it either as $\drankvar{M}{E}$ for some cogenerator $E$, or as $\rankvar{M}{E'}$ for some generator $E'$.

\begin{exa}
\label{KP-eg}
As mentioned in the introduction, a precursor to Theorem~\ref{affine-quot}(2) can be found in work of Kraft and Procesi \cite[\S3]{KP}. They prove this result in a very special case, which we will now describe, expanding on the explanation in \cite[\S8.3]{CBSa}.

Kraft and Procesi's setting concerns the Zariski closure of the conjugacy class of a nilpotent ($n\times n$)-matrix $\lambda$. Letting $A=\K[t]/(t^n)$ be the truncated polynomial ring, $\lambda$ determines an $n$-dimensional $A$ module $M$ with underlying vector space $\K^n$, on which $t$ acts as multiplication by $\lambda$. The conjugacy class of $\lambda$ is precisely the orbit $\orbit{M}\subset\rep{A}{n}$. Since $A$ is a representation-finite algebra, we have $\clorbit{M}=\drankvar{M}{E}$ for $E$ an additive generator of $\lmod{A}$, and so this variety fits into the setting of Theorem~\ref{affine-quot} (and indeed of \cite[Thm.~7.4]{CBSa}), allowing us to realise it as an affine quotient, as follows.

The indecomposable $A$-modules are, up to isomorphism, $M_i=\K[t]/(t^i)$ for $1\leq i\leq n$, with $M_n$ being the unique indecomposable projective and unique indecomposable injective. Note that $\dim{M_i}=i$, making it convenient for us to also set $M_0=0$. Taking $E=\bigoplus_{i=1}^nM_i$ to be the basic additive generator of $\lmod{A}$, the associated cogenerator tilted algebra is
\[B=\End_{\homot[b]{A}}(Q^E)^{\op},\quad Q^E=\Big(\bigoplus_{i=0}^n(M_i\to M_n)\Big)\]
where, for $i\geq 1$, the map $M_i\to M_n$ is the inclusion sending the generator $1\in M_i$ to $t^{n-i}\in M_n$. (Below, we denote maps of indecomposable $A$-modules by the image of the generator $1$ of the domain, so that the preceding inclusion is simply denoted $t^{n-i}$.) Note that we could exclude the summand $M_n\stackrel{1}{\to} M_n$ from $Q^E$, since it is zero in the homotopy category, but we write it here to make the notation $Q^E$ consistent with that earlier in the paper.

The algebra $B$ is presented by the quiver
\[\begin{tikzcd}
(n-1)\arrow[bend left]{r}&\cdots\arrow[bend left]{l}\arrow[bend left]{r}&2\arrow[bend left]{l}\arrow[bend left]{r}&1\arrow[bend left]{l}\arrow[bend left]{r}&0\arrow[bend left]{l}
\end{tikzcd}\]
with relations as follows: for each vertex $1\leq i\leq n-2$, the two $2$-cycles starting at $i$ are equal, and the unique $2$-cycle starting at $(n-1)$ is zero; cf.\ \cite[p.~232]{KP}. Note that there is no relation involving the $2$-cycle starting at vertex $0$. Thus in this case $B$ is not only derived equivalent to the Auslander algebra $\End_A(E)^{\op}$ of $A$, but even isomorphic to it. In this presentation, vertex $i$ corresponds to the summand $M_i\to M_n$ of $Q^E$, and the arrows $(i-1)\to i$ and $i\to(i-1)$ correspond to the morphisms
\[\begin{tikzcd}
M_{i}\arrow{r}{t^{n-i}}\arrow{d}{1}&M_n\arrow{d}{t}\\
M_{i-1}\arrow{r}{t^{n-(i-1)}}&M_n
\end{tikzcd}\qquad
\begin{tikzcd}
M_{i-1}\arrow{r}{t^{n-(i-1)}}\arrow{d}{t}&M_n\arrow{d}{1}\\
M_{i}\arrow{r}{t^{n-i}}&M_n
\end{tikzcd}\]
respectively, so one can check for example that the $2$-cycle starting at $(n-1)$ corresponds to the morphism
\[\begin{tikzcd}
M_{n-1}\arrow{r}{t}\arrow{d}{t}&M_n\arrow{d}{t}\\
M_{n-1}\arrow{r}{t}&M_n
\end{tikzcd}\]
which is null-homotopic, as claimed. We have $A\cong e_0Be_0$ as expected; $e_0$ is the idempotent corresponding to the summand $(M_0\to M_n)=(0\to\kdual A)$ of $Q^E$.

The conjugacy classes of nilpotent $(n\times n)$-matrices, or equivalently the $\Gl_n$-orbits of $n$-dimensional $A$-modules, are parameterised by partitions of $n$, i.e.\ tuples $p=(p_1,\dotsc,p_n)$ of non-negative integers with $p_1\geq p_2\geq\cdots\geq p_n$ and $\sum_{j=1}^np_j=n$; note that we insist that our partitions always have $n$ elements, which we achieve by allowing some of these elements to be $0$. Precisely, $p$ indexes the conjugacy class of the direct sum of Jordan blocks with eigenvalue $0$ and sizes $p_j$, or equivalently the orbit $\orbit{M}$ of $M=\bigoplus_{j=1}^nM_{p_j}$.

To realise $\clorbit{M}$ as an affine quotient via Theorem~\ref{affine-quot}, we need to compute the dimension vector of the $B$-module $c(M)$. By Lemma~\ref{c-description}, we have
\[c(M_j)(M_i\to M_n)=\Ker\big(\kdual\Hom_A(M_j,M_n)\to\kdual\Hom_A(M_j,M_i)\big),\]
where the map is obtained by applying $\kdual\Hom_A(M_j,-)$ to the inclusion $M_i\to M_n$ appearing as a summand in $Q^E$. Thus $\dim\big(c(M_j)(M_i\to M_n)\big)=\max\{j-i,0\}$, and so
\[\dim\big(c(M)(M_i\to M_n)\big)=\sum_{j=1}^n\max\{p_j-i,0\}=:d_{i}.\]
Then by Theorem~\ref{affine-quot} we have an isomorphism
\[\rep{B}{d}\gitquot\Gl_{d'}\isoto\clorbit{M},\]
where $d'=(d_1,\dotsc,d_{n-1})$, given by restriction to vertex $0$.

We claim that this realisation of $\clorbit{M}$ as an affine quotient variety coincides exactly with that obtained by Kraft and Procesi in \cite[\S3]{KP}. Consider the dual partition $\hat{p}$ with $\hat{p}_j=\#\{k:p_k\geq j\}$ for $1\leq j\leq n$, and write $u_i=\hat{p}_{i+1}+\dotsc+\hat{p}_n$ for $0\leq i\leq n-1$. Then Kraft and Procesi's isomorphism is
\[\rep{B}{u}\gitquot\Gl_{u'}\isoto\clorbit{M},\]
where $u'=(u_1,\dotsc,u_{n-1})$, again given by restriction to vertex $0$; that the variety denoted by $Z$ in \cite{KP} is precisely $\rep{B}{u}$ can be seen directly from its description on \cite[p.~232]{KP}, noting that we have reversed the indexing in our description of the dimension vector, so our $u_0$ is the dimension of their $U_t$. However, we have
\[d_i=\sum_{j=1}^n\max\{p_j-i,0\}=\sum_{j=i+1}^n\#\{k:p_k\geq j\}=u_i.\]
To see the middle equality, note that if one draws the partition $p$ as a Young diagram with $p_i$ blocks in the $i$-th row, both sides count the number of blocks in columns $i+1$ to $n$. Thus $d=u$, and the two realisations coincide as claimed.

Kraft and Procesi use this result to prove that the orbit closure $\clorbit{M}$ is a normal variety, by first using special properties of the algebra $B$ appearing in this example to conclude that $\rep{B}{d}$ is normal.
\end{exa}

\subsection{Desingularisation of orbit closures for gen-finite modules}
As above, let $A$ be a finite-dimensional algebra, and let $B$ be any finite-dimensional algebra possessing an idempotent $e$ with $A\cong eBe$. Recall that this data induces a recollement
\[\begin{tikzcd}[column sep=40pt]
\lmod{B/BeB}\arrow["i" description]{r}&\lmod{B}\arrow["e" description]{r}\arrow[shift right=3,swap]{l}{q}\arrow[shift left=3]{l}{p}&\lmod{A}\arrow[shift right=3,swap]{l}{\ell}\arrow[shift left=3]{l}{r}
\end{tikzcd}\]
as in \eqref{recollement}.

We call the $B$-modules in $\cogen (\kdual (eB))$ \emph{stable}, and those in $\gen (Be)$ \emph{costable}; note that the category of stable modules is closed under taking submodules, and the category of costable modules is closed under taking quotients. By Lemma~\ref{ImageAndKernel}, the image of the intermediate extension functor $c$ corresponding to $e$ is the category of modules which are both stable and costable. For any subset $Z\subseteq \rep{B}{d,s}$ we write $Z^{\st}$ for the set of stable modules in $Z$, and $Z^{\cost}$ for the set of costable modules in $Z$.

The action of $\Gl_s$ on $\rep{B}{d,s}^{\st}$ admits a geometric quotient, and \cite[\S6.3]{CBSa} (following \cite{Ki}) use geometric invariant theory to construct a projective map
\[
\pi \colon \rep{B}{d,s}^{\st} / \Gl_s \to  \rep{B}{d,s}\gitquot \Gl_s \isoto \Bild e
\]
from this geometric quotient to the categorical quotient of the $\Gl_s$-action on the whole representation variety. Now let $M$ be an $A$-module. Choosing a cogenerator $E$ such that $\drankvar{M}{E}=\drankvar{M}{}$, letting $B$ be the cogenerator-tilted algebra of $E$ with special idempotent $e$, and setting $(d,s)=\Dim{c(M)}$, it follows from Theorem~\ref{affine-quot*} that $\Bild{e}=\drankvar{M}{}$, so the above construction gives a projective map $\pi\colon\rep{B}{d,s}^{\st}\to\drankvar{M}{}$. Our aim for the remainder of the section is to show that, when $M$ is gen-finite (so $\drankvar{M}{}=\clorbit{M}$ by Theorem~\ref{rankvar=clorbit}) and $\gen(M)\subseteq\add{E}$, this map $\pi$ above restricts to a desingularisation
\[\pi\colon\clorbit{c(M)}^{\st}/\Gl_s\to\clorbit{M}.\]

We begin with the following lemma, dual to a statement of Zwara \cite[Proof of Thm.~1.2(1)]{Zw2}; for convenience, we give a complete argument.

\begin{lem}\label{degenerate}
Let $\Lambda $ be a finite-dimensional algebra and
\begin{equation}
\label{degen-ses-1}
\begin{tikzcd}[column sep=20pt]
0\arrow{r}&X\arrow{r}&M\oplus Z\arrow{r}&Z\arrow{r}&0
\end{tikzcd}
\end{equation}
a short exact sequence of $\Lambda$-modules. Then there exists an exact sequence
\[\begin{tikzcd}[column sep=20pt]
0\arrow{r}& X \arrow{r}& M\oplus Z' \arrow{r}& Z'  \arrow{r}& 0
\end{tikzcd}\]
with $Z' \in \gen{M}$.
\end{lem}

\begin{proof}
Let $\begin{tikzcd}[column sep=20pt,ampersand replacement=\&]
0\arrow{r}\& X \arrow{r}\& M\oplus Z' \arrow{r}{\left(\begin{smallmatrix}\delta&\gamma\end{smallmatrix}\right)}\& Z'  \arrow{r}\& 0\end{tikzcd}$
be the short exact sequence obtained by splitting off a maximal direct summand of the form
$\smash{\begin{tikzcd}[ampersand replacement=\&,column sep=20pt]
0\arrow{r}\& 0\arrow{r}\& Z'' \arrow{r}{\id}\& Z''\arrow{r}\& 0
\end{tikzcd}}$
from \eqref{degen-ses-1}, so that $\gamma\in\End_{\Lambda}(Z')$ is nilpotent. We claim that then $Z'\in\gen(M)$, so this is our desired sequence. By induction on $k$, one proves that
\[ 
( \delta,  \gamma \delta ,  \gamma^2 \delta , \ldots , \gamma^{k-1} \delta  , \gamma^k)\colon M^k \oplus Z' \to Z'
\] 
is surjective for any $k \geq 1$. When $k\gg0$ we have $\gamma^k=0$, and thus the map $(\delta,\gamma\delta,\dotsc,\gamma^{k-1}\delta)\colon M^k \to Z'$ is an epimorphism.
\end{proof}

The main step in our argument is the following theorem, characterising the stable modules in $\clorbit{c(M)}$ and giving a sufficient condition for them to be smooth points of this variety.

\begin{thm} \label{Zwara} Let $A$ and $B$ be basic algebras with $A\cong eBe$ for some idempotent $e$. Let $\widetilde{N}\in \rep{B}{d,s} $ and write $N=e\widetilde{N}\in \rep{A}{d}$. Then $\widetilde{N} \in \clorbit{c(M)}^{\st}$ for $M\in\rep{A}{d}$ if and only if there is an exact sequence
\begin{equation}
\label{degen-seq}
\begin{tikzcd}[column sep=20pt]0\arrow{r}& N\arrow{r}&M\oplus Z \arrow{r}{p} &Z\arrow{r}&0\end{tikzcd}
\end{equation}
such that $\widetilde{N}\cong \Ker c(p) $. 
If moreover $c(M\oplus Z)$ is rigid, then $\widetilde{N}$ is a smooth point of $\clorbit{c(M)}^{\st} $.  
\end{thm}

\begin{proof}
If there is an exact sequence \eqref{degen-seq}, applying $c$ to it gives an exact sequence
\[\begin{tikzcd}[column sep=20pt]
0\arrow{r}& \widetilde{N} \arrow{r}& c(M) \oplus c(Z) \arrow{r}& c(Z) \arrow{r}& 0.
\end{tikzcd}\]
By \cite[Thm.~1]{Zw}, this implies $\widetilde{N} \in \clorbit{c(M)}$. Since modules in the image of $c$ are stable, and $\widetilde{N}$ is a submodule of such a module, it is also stable.

Conversely, assume $ \widetilde{N} \in \clorbit{c(M)}^{\st}$. Using \cite[Thm.~1]{Zw} again, we may find an exact sequence
\[\begin{tikzcd}[column sep=20pt]
0\arrow{r}& \widetilde{N} \arrow{r}& c(M) \oplus \widetilde{Z}   \arrow{r}{\widetilde{p}}& \widetilde{Z} \arrow{r}& 0.
\end{tikzcd}\]
By Lemma~\ref{degenerate} we may choose this exact sequence so that $\widetilde{Z}\in\gen(c(M))$, and so $\widetilde{Z}$ is costable. Applying $e$ to this sequence, and writing $p=e(\widetilde{p})$ and $Z=e\widetilde{Z}$, we obtain
\[\begin{tikzcd}[column sep=20pt]
0\arrow{r}&N\arrow{r}&M\oplus Z\arrow{r}{p}&Z\arrow{r}&0.
\end{tikzcd}\]
We claim that this is our desired sequence \eqref{degen-seq}, i.e.\ that $\widetilde{N}\cong\ker{c(p)}$. Since any costable $B$-module $X$ has a natural epimorphism $q\colon X\to ce(X)$ \cite[Lem.~2.4]{CBSa}, there is a commutative diagram 

\[\begin{tikzcd}[ampersand replacement=\&]
0\arrow{r}\&\widetilde{N}\arrow{r}\&c(M)\oplus \widetilde{Z}\arrow{d}{\left(\begin{smallmatrix}1&0\\0&q\end{smallmatrix}\right)}\arrow{r}{\widetilde{p}}\&\widetilde{Z}\arrow{r}\arrow{d}{q}\&0\\
0\arrow{r}\&\Ker{c(p)}\arrow{r}\&c(M)\oplus c(Z)\arrow{r}{c(p)}\&c(Z)\arrow{r}\&0.
\end{tikzcd}\]
This induces a morphism $\varphi \colon \widetilde{N} \to \Ker c(p)$, and since $e(q)$ is an isomorphism, so is $e(\varphi)$. Now the unit $\varepsilon\colon\id\to re$ induces a commutative diagram

\[\begin{tikzcd}[column sep=30pt]
\widetilde{N}\arrow{r}{\varepsilon_{\widetilde{N}}}\arrow{d}{\varphi}&r(N)\arrow{d}{re(\varphi)}\\
\Ker{c(p)}\arrow{r}{\varepsilon_{\ker{c(p)}}}&re(\Ker{c(p)})
\end{tikzcd}\]
and stability of $\widetilde{N}$ implies \cite[Lem.~2.3]{CBSa} that $\varepsilon_{\widetilde{N}}$ is a monomorphism. Moreover, $re(\varphi)$ is an isomorphism (since $e(\varphi)$ is), and so $\varphi $ is a monomorphism. But $\Dim \widetilde{N} = \Dim{c(M)}=\Dim \Ker c(p)$, so $\varphi $ is an isomorphism as required.

It remains to show that if we have a sequence \eqref{degen-seq} such that $c(M\oplus Z)$ is rigid, then $\widetilde{N}$ is a smooth point of $\clorbit{c(M)}^{\st}$. Since the stable locus is open, it suffices to show that $\widetilde{N}$ is a smooth point of $\clorbit{c(M)}$. Since we have an exact sequence
\begin{equation}
\label{eq:smoothpt}
\begin{tikzcd}[column sep=20pt]
0\arrow{r}& \widetilde{N} \arrow{r}& c(M \oplus Z) \arrow{r}& c(Z) \arrow{r}& 0,
\end{tikzcd}
\end{equation}
it follows from the dual of \cite[Prop.~2.2]{Zw2} that it is even enough to show
\[\dim\Hom_B(\widetilde{N}, c(M \oplus Z)) =\dim\Hom_B(c(M) , c(M \oplus Z)).\]
But applying $\Hom_B(-, c(M\oplus Z))$ to the sequence \eqref{eq:smoothpt} yields
\[\begin{tikzcd}[column sep=17pt]
0\arrow{r}&\Hom_B(c(Z),c(M\oplus Z))\arrow{r}&\Hom_B(c(M\oplus Z),c(M\oplus Z))\arrow{r}&\Hom_B(\widetilde{N},c(M\oplus Z))\arrow{r}&0,
\end{tikzcd}\]
since $\Ext^1_B(c(Z),c(M\oplus Z))=0$ by the rigidity of $c(M\oplus Z)$, and so the desired equality follows.
\end{proof}

We are now ready to describe our desingularisation for the orbit closure of a gen-finite module.

\begin{thm}\label{gen-finite-desing}
Assume $M\in\lmod{A}$ is gen-finite. Let $E$ be a basic cogenerator with $\gen(M)\subseteq \add E$, and let $B$ be the cogenerator-tilted algebra of $E$ with special idempotent $e$. For $c\colon\lmod{A}\to\lmod{B}$ the intermediate extension functor corresponding to $e$, write $(d,s)=\Dim{c(M)}$, and let $\pi\colon\rep{B}{d,s}^{\st}/\Gl_s\to\Bild{e}$ be the projective map constructed in \cite[\S6.3]{CBSa}. Then the restriction
\[\pi \colon \clorbit{c(M)}^{\st} / \Gl_s \to \clorbit{M}\]
is a desingularisation with connected fibres. 
\end{thm}

\begin{proof}
First note that by Theorem~\ref{rankvar=clorbit} and Theorem~\ref{affine-quot} we have $\Bild{e}=\drankvar{M}{}=\clorbit{M}$, so the codomain is as claimed. Since $\clorbit{c(M)}^\st\subseteq\rep{B}{d,s}^{\st}$ is a closed subscheme, the restricted map is still projective.

Next we prove that $\clorbit{c(M)}^{\st}$ is smooth. Let $\widetilde{N} \in \clorbit{c(M)}^{\st}$, write $N=e\widetilde{N}$ and take an exact sequence
\[\begin{tikzcd}[column sep=20pt]
0\arrow{r}&N\arrow{r}&M\oplus Z\arrow{r}&Z\arrow{r}&0
\end{tikzcd}\]
as in Theorem~\ref{Zwara}(2). By Lemma~\ref{degenerate}, we may choose this sequence so that $Z\in\gen(M)\subseteq\add{E}$. Thus $c(M\oplus Z)\subseteq\add{c(E)}$ is rigid by Theorem~\ref{dual-projres}, and by Theorem~\ref{Zwara} again we see that $\widetilde{N}$ is smooth.

By \cite[Prop.~4.5(2)]{ES}, the functor $c$ induces an isomorphism $\orbit{M} \isoto \pi^{-1} (\orbit{M})= \orbit{c(M)}/\Gl_s$ inverse to $\pi|_{\orbit{c(M)}}$, so $\pi$ is an isomorphism over $\orbit{M}$, and hence a desingularisation. Finally, since $\clorbit{M}$ is unibranch \cite[Thm.~1.2(3)]{Zw2}, the fibres of any desingularisation are connected \cite[Lem.~4.1(1)]{Zw2}.
\end{proof}

By dualising the argument, we may also construct a desingularisation of $\clorbit{M}$ when $M$ is cogen-finite, by using a generator $E$ such that $\rankvar{M}{E}=\drankvar{M}{}=\clorbit{M}$, and using Theorem~\ref{affine-quot*} to express this as an affine quotient variety. We leave it to the reader to dualise Lemma~\ref{degenerate} and Theorem~\ref{Zwara}, and state only the dual of Theorem~\ref{gen-finite-desing}.

\begin{thm}[Dual to Theorem~\ref{gen-finite-desing}]\label{cogen-finite-desing}
Assume $M\in\lmod{A}$ is cogen-finite. Let $E$ be a basic generator with $\cogen(M)\subseteq \add E$, and let $B$ be the generator-cotilted algebra of $E$ with special idempotent $e$. For $c\colon\lmod{A}\to\lmod{B}$ the intermediate extension functor corresponding to $e$, write $(d,s)=\Dim{c(M)}$. One can construct a map $\pi\colon\rep{B}{d,s}^{\cost}/\Gl_s\to\Bild{e}$ dual to that of \cite[\S6.3]{CBSa}. Then the restriction
\[\pi \colon \clorbit{c(M)}^{\cost} / \Gl_s \to \clorbit{M}\]
is a desingularisation with connected fibres. 
\end{thm}

\section{Desingularisation of quiver Grassmannians}
\label{quiver-Grassmannians}

In this section we take $\K$ to be an algebraically closed field, for compatibility with \cite[\S7]{CBSa}, and let $A=\K Q/I$ be a finite-dimensional algebra presented by a finite quiver $Q$ and an admissible ideal $I$. Let $M$ be a finite-dimensional $A$-module and $d\in \Z_{\geq0}^{Q_0}$ a dimension vector.  Our aim in this section is to desingularise the quiver Grassmannian $\Gra{A}{M}{d}$ of $d$-dimensional submodules of $M$, in the case that $M$ is gen-finite. The majority of this section is devoted to proving the following theorem, from which we obtain smooth varieties to use in constructing our desingularisation.

\begin{thm}
\label{Gr-smooth}
Let $M$ be a gen-finite module and let $E$ be the cogenerator given by the direct sum of all indecomposable modules in $\gen(M)$ together with any remaining indecomposable injectives. Let $B$ be the cogenerator-tilted algebra of $E$ with special idempotent $e$, let $c$ be the intermediate extension associated to $e$, and let $(d,s)$ be a dimension vector for $B$. Then if the Grassmannian $\Gra{B}{c(M)}{d,s}$ is non-empty, it is (scheme-theoretically) smooth and equidimensional. 
\end{thm}

Before proving this theorem, we require some preparation.

\begin{lem}
\label{Gr-lem1}
Let $B$ be a finite-dimensional basic algebra and $X$ a finite-dimensional $B$-module with $\idim X \leq 1$ and $\Ext^1_{B} (X,X) =0$. If $d$ is such that $\Gra{B}{X}{d}\ne\varnothing$ and $\Ext^1_B(U, X/U)=0$ for every $U \in \Gra{B}{X}{d }$, then $\Gra{B}{X}{d }$ is smooth and equidimensional. 
\end{lem}

\begin{proof}
We first give a convenient realisation of $\Gra{B}{X}{d}$. Let $\Lambda=\left(\begin{smallmatrix}B&B\\0&B\end{smallmatrix}\right)$, so that a $\Lambda $-module is precisely the data of a $B$-linear morphism between two $B$-modules. Let $r := \Dim X$, and consider the representation space $\rep{\Lambda}{d,r}$, parametrising morphisms of $B$-modules with $d$-dimensional domain and $r$-dimensional codomain.

Note that since $\Ext^1_B(X,X)=0$, the orbit $\orbit{X}\subset\rep{B}{r}$ is open, and thus so is its preimage under the canonical projection $\pi:\rep{\Lambda}{d,r}\to\rep{B}{r}$. Let $H$ be the intersection of $\pi^{-1}(\orbit{X})$ with the open subset of $\rep{\Lambda}{d,r}$ consisting of monomorphisms, so that $H$ is again open. Moreover, $H$ is a principal $\Gl_d$-fibre bundle over the set $\underline{H}$ of points $Z\xrightarrow{f}Y$ of $\rep{\Lambda}{d,r}$ where $f$ is a set-theoretic inclusion; we denote such a point by $(Z\subset Y)$. Explicitly, the bundle is defined by $p\colon H\to \underline{H}$ with $p(Z\xrightarrow{f} Y)= (\Bild f \subset Y)$. By construction we have a projective $\Gl_{r}$-equivariant map $\pi \colon \underline{H} \to \orbit{X}$, $(Z\subset Y) \mapsto Y$, with fibre over $X$ given by $\pi^{-1}(X)=\Gra{B}{X}{d}$. Consider the multiplication map $\Gl_{r}\times\Gra{B}{X}{d}\to\underline{H}$ given by $(g,U\subset X)\mapsto (gU\subset gX)$. The composition of this map with $\pi$ induces surjective maps on tangent spaces, since the map $\Gl_{r}\to \orbit{X}:g\mapsto gX$ has this property, and hence so does $\pi$. Thus we obtain a short exact sequence
\begin{equation}
\label{eq:tangents}
\begin{tikzcd}[column sep=20pt]
0\arrow{r}&\tang{(U\subset X)}{\Gra{B}{X}{d}}\arrow{r}&
\tang{(U\subset X)}{\underline{H}}\arrow{r}&\tang{X}{\orbit{X}}\arrow{r}&0.
\end{tikzcd}
\end{equation}
of tangent spaces.

To deduce the desired properties of $\Gra{B}{X}{d}$ from \eqref{eq:tangents}, we first show that $\underline{H}$ is smooth. Since $H$ is a principal $\Gl_d$-bundle over $\underline{H}$, we may instead show that $H$ is smooth. A $\Lambda$-module $\widehat{U}$ is a smooth point of $H$ if and only if it is smooth in $\rep{\Lambda}{d,r}$, since $H$ is an open subset, and by \cite[Lem.~6.4]{CBSa} this is the case whenever $\Ext^2_\Lambda(\widehat{U},\widehat{U})=0$.

So let $\widehat{U}=(U\hookrightarrow X)$ be a point of $H$. We also consider the $\Lambda$-module $\widehat{X}=(X\xrightarrow{1_X}X)$, and note that there is a natural inclusion $\widehat{U}\to\widehat{X}$ such that $\widehat{X}/\widehat{U}=(X/U\to 0)$. Applying $\Hom_{\Lambda}(\widehat{U},-)$ to the short exact sequence
\[\begin{tikzcd}[column sep=20pt]
0\arrow{r}&\widehat{U}\arrow{r}&\widehat{X}\arrow{r}&\widehat{X}/\widehat{U}\arrow{r}&0
\end{tikzcd}\]
yields
\begin{equation}
\label{les}
\begin{tikzcd}[column sep=20pt]
\cdots \arrow{r}& \Ext^1_{\Lambda} (\widehat{U}, \widehat{X} ) \arrow{r}& \Ext_{\Lambda} ^1(\widehat{U}, \widehat{X}/\widehat{U}) \arrow{r}& \Ext^2_{\Lambda} (\widehat{U},\widehat{U}) \arrow{r}& \Ext^2_{\Lambda} (\widehat{U} , \widehat{X} ) =0,
\end{tikzcd}
\end{equation}
where the last space is zero because $\idim_{\Lambda} \widehat{X} = \idim_{B} X \leq 1$.

We now claim that $\Ext^1_{\Lambda}(\widehat{U},\widehat{X}/\widehat{U}) \cong \Ext^1_{B}(U, X/U)=0$. To see this, chose an exact sequence
\begin{equation}
\label{ses}
\begin{tikzcd}[column sep=20pt]
0\arrow{r}& X/U \arrow{r}& Q \arrow{r}& Y \arrow{r}& 0
\end{tikzcd}
\end{equation}
with $Q$ injective. This induces an exact sequence
\[\begin{tikzcd}[column sep=20pt]
0 \arrow{r}& \widehat{X}/\widehat{U} \arrow{r}& (Q\to 0) \arrow{r}& (Y\to 0) \arrow{r}& 0
\end{tikzcd}\]
of $\Lambda$-modules, to which we apply $\Hom_{\Lambda}(\widehat{U},-)$ to obtain 
\[\begin{tikzcd}[column sep=20pt]
 0\arrow{r}& \Hom_{\Lambda}(\widehat{U},\widehat{X}/\widehat{U})\arrow{r}& \Hom_{\Lambda}(\widehat{U}, Q\to 0) \arrow{r}& \Hom_{\Lambda}(\widehat{U}, Y\to 0 ) \arrow{r}& \Ext^1_{\Lambda}(\widehat{U},\widehat{X}/\widehat{U}) \arrow{r}& 0,
\end{tikzcd}\]
using that $(Q\to0)$ is an injective $\Lambda$-module. Applying $\Hom_{B}(U, -)$ to \eqref{ses}, we obtain
\[\begin{tikzcd}[column sep=20pt]
 0\arrow{r}& \Hom_{B}( U, X/U)\arrow{r}& \Hom_{B}(U, Q) \arrow{r}& \Hom_{B}(U,Y) \arrow{r}& \Ext^1_B(U,X/U) \arrow{r}& 0.
\end{tikzcd}\]
The first three terms of the preceding four-term exact sequences are isomorphic, therefore also the fourth, i.e.\ $\Ext^1_{\Lambda}(\widehat{U},\widehat{X}/\widehat{U}) \cong \Ext^1_{B}(U, X/U)$, and $\Ext^1_{B}(U, X/U)=0$ by assumption. Looking back at \eqref{les}, we see that
$\Ext^2_{\Lambda} (\widehat{U},\widehat{U}) =0$, and conclude that $\underline{H}$ is smooth as above.

It now follows from \eqref{eq:tangents} that $\dim{\tang{(U\subset X)}{\Gra{B}{X}{d}}}=\dim{\underline{H}}-\dim{\mcO_X}$ and so $\Gra{B}{X}{d}$ is equidimensional. Recalling that $\Gra{B}{X}{d}$ is a fibre of the projective map $\pi\colon\underline{H}\to\orbit{X}$, it follows from Chevalley's theorem (see for example \cite[Thm.~5.26]{Wolf}) that
\[\dim{\tang{(U\subset X)}{\Gra{B}{X}{d}}}\geq\dim{\Gra{B}{X}{d}}\geq\dim{\underline{H}}-\dim{\orbit{X}},\]
and from the equality of the outer terms we conclude that $\Gra{B}{X}{d}$ is smooth.
\end{proof}

\begin{lem}
\label{Gr-lem2}
Let $B$ be a finite-dimensional basic algebra and $X\in\lmod{B}$. Assume $C\in\lmod{B}$ satisfies $\Ext^i_B(C,C)=0$ for all $i>0$, and that $X\in\add{C}$. Then if $U \in \Gra{B}{X}{d}$ fits into a short exact sequence
\[\begin{tikzcd}[column sep=20pt]
0\arrow{r}& U \arrow{r}& C_0 \arrow{r}& C_1 \arrow{r}& 0
\end{tikzcd}\]
with $C_i \in \add{C}$, we have $\Ext^i_B(U,X/U)=0$ for all $i>0$. In particular, if every $U\in\Gra{B}{X}{d}$ fits into such a sequence, then $\Gra{B}{X}{d} $ is smooth and equidimensional (provided it is not empty).
\end{lem}

\begin{proof} We begin by establishing two intermediate facts.
\begin{itemize}
\item[(i)]Applying $\Hom_B(-,C)$ to the short exact sequence
\[\begin{tikzcd}[column sep=20pt]
0\arrow{r}& U \arrow{r}& C_0 \arrow{r}& C_1\arrow{r}& 0,
\end{tikzcd}\]
we see that $\Ext^i_B(U,C)$ for $i\geq1$.
\item[(ii)]It then follows that $\Ext^i_B(U,U)=0$ for $i\geq2$ by applying $\Hom_B(U,-)$ to the same exact sequence and using (i).
\end{itemize}
Now apply $\Hom_B(U,-)$ to the short exact sequence 
\[\begin{tikzcd}[column sep=20pt]
0\arrow{r}& U \arrow{r}& X \arrow{r}& X/U \arrow{r}& 0
\end{tikzcd}\]
to obtain 
\[\begin{tikzcd}[column sep=20pt]
\cdots \arrow{r}& \Ext^i_B(U,X) \arrow{r}& \Ext^i_B(U, X/U) \arrow{r}& \Ext^{i+1}_B(U,U) \arrow{r}& \cdots 
\end{tikzcd}\]
for each $i>0$. We have $\Ext^i_B(U,X)=0$ by (i), since $X\in\add{C}$, and $\Ext^{i+1}_B(U,U)=0$ by (ii). Thus we conclude $\Ext^i_B(U,X/U)=0$.  The final conclusion is then a direct application of Lemma~\ref{Gr-lem1}.
\end{proof}

\begin{rem}
By a result of Wolf \cite[Lem.~5.22]{Wolf}, the tangent space to $\Gra{B}{X}{d}$ at a point $U$ may be identified with $\Hom_B(U,X/U)$. Thus, under the assumptions of Lemma~\ref{Gr-lem1}, $\dim\Hom_B(U,X/U)=\dim\Gra{B}{X}{d}$ is independent of $U$. Under the stronger assumptions of Lemma~\ref{Gr-lem2}, it follows that $\Ext^i_B(U,X/U)=0$ for all $i>0$, so when $\gldim{B}<\infty$ we also have
\[\dim\Hom_B(U,X/U)=\langle d,\Dim{X}-d\rangle,\]
where $\langle-,-\rangle$ denotes the Euler form of $B$ (cf.\ \cite[Cor.~3]{CR}).
 \end{rem}

We are now ready to prove Theorem~\ref{Gr-smooth}. 
\begin{proof}[Proof of Theorem~\ref{Gr-smooth}]
By Theorem \ref{int-ext-calcs}, $C= c(E)$ is the canonical cotilting module for $B$. In particular, $\Ext^i_B(C,C)=0$ for all $i>0$, and $c(M)\in\add{C}$ since $M\in\add{E}$. Let $U \in \Gra{B}{c(M)}{d,s}$. We claim that there is an exact sequence
\[\begin{tikzcd}[column sep=20pt]
0\arrow{r}& U \arrow{r}& C_0 \arrow{r}& C_1 \arrow{r}& 0
\end{tikzcd}\]
with $C_0, C_1 \in \add{C}$, so that the result follows by Lemma~\ref{Gr-lem2}.

We write $E=N \oplus Q$ with $\add{N}=\gen(M)$ and $Q$ injective, and let $f\colon U \to C_0$ be a left $\add(c(N))$-approximation of $U$. Since $U \subseteq c(M)$ with $c(M) \in \add c(N)$, and this inclusion must factor over $f$, we see that $f$ is a monomorphism. We complete it to the short exact sequence 
\begin{equation}
\label{Gr-smooth-ses}
\begin{tikzcd}[column sep=20pt]
0\arrow{r}& U \arrow{r}{f}& C_0 \arrow{r}& C_1\arrow{r}& 0,
\end{tikzcd}
\end{equation}
in which $C_0\in \add{C}$ by construction, and $C_1 \in \add{C}$ as we now show. By applying $\Hom_{B}(-,c(N))$ to \eqref{Gr-smooth-ses} and using that $f$ is an $\add(c(N))$-approximation and $c(N)$ is rigid, we see that $\Ext^1_{B} (C_1, c(N)) =0$. Since $c(Q) $ is injective by Theorem~\ref{int-ext-calcs} we even have $\Ext^1_{B} (C_1, C)=0$. As $C $ is a cotilting module, this means $C_1 \in \cogen (C)$, but we also have $C_1 \in \gen (C)$ by \eqref{Gr-smooth-ses}.

We claim that $\gen (C) \cap \cogen (C) \subseteq \Bild c$. To see this, note that whenever we have
$c(X) \xrightarrow{p} Y \xrightarrow{i} c(Z)$ with $p$ an epimorphism and $i$ a monomorphism, we have $ip=c(g)$ for some morphism $g\colon X\to Z$ since $c$ is fully faithful. But $c$ preserves epimorphisms and monomorphisms, so $Y\cong \Bild c(g) \cong c(\Bild g) \in \Bild c$.

Thus $C_1 \in \Bild c$, and so $C_1=ceC_1$. Since $C_0\in\add{c(N)}$, we have $eC_0\in\add{N}=\gen(M)$. Hence $eC_1\in\gen(M)\subseteq\add{E}$. Recalling that $C=c(E)$, we conclude that $C_1 =ceC_1 \in \add{C}$, completing the proof.
\end{proof}

We now give the construction of our promised desingularisation for the quiver Grassmannian $\Gra{A}{M}{d}$ of a gen-finite $A$-module $M$. Let $M$ be such a module. For $N\in\lmod{A}$ we write 
\[ 
\grstrat{N}:= \{ U \in \Gra{A}{M}{d}\mid M/U\cong N \}.
\]
This is a locally closed irreducible subset of $\Gra{A}{M}{d}$. 
Since $M$ is gen-finite, there is finite set of modules $N_1, \ldots ,N_t$ such that $\Gra{A}{M}{d} = \bigcup_{i=1}^t\clgrstrat{N_i}$. It follows that the irreducible components of $\Gra{A}{M}{d}$ must be among the closed sets $\clgrstrat{N_i}$, so without loss of generality we may assume that $N_1,\dotsc,N_t$ were chosen such that $\clgrstrat{N_1},\dotsc,\clgrstrat{N_t}$ are precisely these components.

As in Theorem~\ref{Gr-smooth}, let $E$ be the cogenerator given by the direct sum of indecomposables in $\gen(M)$ together with any remaining indecomposable injectives. As usual, let $B$ be the cogenerator-tilted algebra of $E$, with special idempotent $e$ and associated intermediate extension $c$. Write $(d,s_i)= \Dim c(M) - \Dim c(N_i)$, which has positive components since $c$ preserves epimorphisms, and consider the algebraic map
\[ 
\Gra{B}{c(M)}{d,s_i} \to \Gra{A}{M}{d}
\]
induced by $e$. This is a projective map since it is an algebraic map between projective varieties, and it restricts to a projective map
\[ 
p_i \colon \clgrstrat{c(N_i)} \to \clgrstrat{N_i}.
\]
Since $\clgrstrat{N_i}$ is an irreducible component of $\Gra{A}{M}{d}$, each $\grstrat{c(N_i)}$ contains a non-empty open subset, so $\clgrstrat{c(N_i)}$ is also an irreducible component---since $\Gra{B}{c(M)}{d,s_i}$ is smooth by Theorem~\ref{Gr-smooth} it is even a connected component. Furthermore, $p_i$ is an isomorphism over an open subset of $\grstrat{N_i}$, by dualising the argument of \cite[Thm.~7.1(3)]{CBSa}.

Combining the various maps $p_i$, we obtain
\[ 
p = \bigsqcup_{i=1}^t p_i\colon \bigsqcup_{i=1}^t \clgrstrat{c(N_i)} \to \Gra{A}{M}{d}.
\]
By Theorem~\ref{Gr-smooth}, the domain of this map is smooth. In summary, we have the following result.
\begin{cor} \label{Gr-desing}
For every gen-finite module $M$, the map $p$ as constructed above is a desingularisation. 
\end{cor}

As usual, we also have a dual construction when $M$ is cogen-finite. For $N\in\lmod{A}$, write
\[\mcS_{[N]}=\{U\in\Gra{A}{M}{d}\mid U\cong N\},\]
so that cogen-finiteness of $M$ implies the existence of a finite set $N_1,\dotsc,N_t$ with $\Gra{A}{M}{d}=\bigcup_{i=1}^t\overline{\mcS}_{[N_i]}$, each $\overline{\mcS}_{[N_i]}$ being an irreducible component. Now let $E$ be the generator given by the direct sum of indecomposables in $\cogen(M)$ together with any remaining indecomposable projectives. Let $B$ be the gen-cotilted algebra of $E$, with special idempotent $e$ and intermediate extension $c$, and write $(d,s_i)=\dim{c(N_i)}$. Then $e$ determines an algebraic map
\[\Gra{B}{c(M)}{d,s_i}\to\Gra{A}{M}{d}.\]
Dualising Theorem~\ref{Gr-smooth}, this map has smooth domain, and so the restriction $p_i\colon\overline{\mcS}_{[c(N_i)]}\to\overline{\mcS}_{[N_i]}$ desingularises the component $\overline{\mcS}_{[N_i]}$ of $\Gra{A}{M}{d}$. Thus taking the disjoint union gives a desingularisation
\[p=\bigsqcup_{i=1}^tp_i\colon\bigsqcup_{i=1}^t\overline{\mcS}_{[c(N_i)]}\to\Gra{A}{M}{d}.\]
\begin{cor}
For every cogen-finite module $M$, the map $p$ constructed above is a desingularisation.
\end{cor}

\section{Gen-finite modules}
\label{s:gen-finite}

Our constructions above, both for orbit closures and quiver Grassmannians, involve the assumption of gen-finiteness of a module. In this section we recall some basic facts about gen-finite modules, including methods for easily constructing examples.
Recall from Definition~\ref{gen-finite} that an $A$-module $M$ is gen-finite if there exists $X\in\lmod{A}$ with $\gen(M)=\add{X}$, and cogen-finite if there exists $X\in\lmod{A}$ with $\cogen(M)=\add{X}$.

\begin{lem} 
If $A=\K Q$ for an acyclic quiver $Q$ and $M$ is a gen-finite $A$-module, then $\tau M$ and $M\oplus \kdual A$ are also gen-finite $A$-modules.
\end{lem}
\begin{proof} 
Since $A$ is hereditary, $\tau^-=\Ext^1_A(\kdual A,-)$ may be defined as a functor $\lmod{A}\to\lmod{A}$ and is left adjoint to $\tau=\kdual\Ext^1_A(-,A)$. Therefore $\tau^-$ preserves epimorphisms, and so
\[\tau^-(\gen (\tau M)) \subseteq \gen (\tau^{-} \tau M) \subseteq \gen (M).\]
We observe that $\gen (\tau M) \subseteq \tau (\tau^- \gen (\tau M)) \oplus \add{\kdual A}$ because every module $N$ can be decomposed as $N= \tau \tau^- N\oplus Q$ for some injective module $Q$. Putting these two observations together we get 
\[ 
\gen (\tau M) \subseteq \tau (\tau^- \gen (\tau M)) \oplus \add{\kdual A} \subseteq \tau (\gen (M)) \oplus \add{\kdual A}
\]
Thus if $M$ is gen-finite, then $\tau M$ is also gen-finite.

Now we prove the second claim. Decompose $Z\in \gen (M\oplus \kdual A)$ into $Z'\oplus Q$ with $Q \in \add{\kdual A}$ maximal, so that $Z'$ has no injective summands. Since $A=\K Q$ is hereditary, it follows that $\Hom_{A}(\kdual A , Z')=0$, and so $Z'$ must be in $\gen(M)$. Thus if $M$ is gen-finite, $M \oplus \kdual A$ is also gen-finite.
\end{proof}

As a corollary, we obtain the following well-known result.
\begin{cor} \label{gen-fin-quiver} If $A=\K Q$ for an acyclic quiver $Q$, then every preinjective module is gen-finite. Also, $\tau^j S$ is gen-finite for every semi-simple module $S$ and every $j\geq0$. 
\end{cor}

More generally, we recall the following definition \cite{Ri3}. 

\begin{dfn} A connected component of the Auslander--Reiten quiver of a finite-dimensional algebra is called \emph{preinjective} if it has no cyclic paths and each of its $\tau^{-}$ orbits contains an injective module. A module is called preinjective if every indecomposable summand is contained in some preinjective component.  
\end{dfn}
For example, an algebra $A$ with the coseparation property \cite[Def.~IX.4.1]{ASS} admits a preinjective component \cite[Thm.~IX.4.5]{ASS}. Zwara proved \cite[Thm.~1.4]{Zw2} that every preinjective module is gen-finite. We recall some other straightforward but useful results.
\begin{pro}
Let $A$ be a finite-dimensional algebra and $M\in\lmod{A}$.
\begin{itemize}
\item[(1)]The module $M$ is gen-finite or cogen-finite as an $A$-module if and only if the corresponding property holds for $M$ treated as an $A/\ann(M)$-module. In particular, if $A/\ann(M)$ is representation-finite, then $M$ is both gen-finite and cogen-finite.
\item[(2)] If $I$ is a $2$-sided ideal in $A$ and $M$ is a gen-finite $A$-module, then $M/IM$ is a gen-finite $A/I$-module. 
\item[(3)] If $e\in A$ is an idempotent such that $P=Ae$ is gen-finite, then $eAe= \End_A(P)^{\op}$ is representation-finite.
\end{itemize}
\end{pro}
\begin{proof}
Statement (1) follows since any $A$-module in $\gen(M)$ or $\cogen(M)$ is also an $A/\ann(M)$-module. Thus we may assume in (2) that $M$ is faithful, so $\ann(M/IM)=I$ and we are in a special case of (1). Finally, (3) follows from Lemma \ref{ImageAndKernel} since $\lmod{eAe}$ is equivalent to $\gen_1(P)\subseteq\gen(P)$ via the left adjoint $\ell $ to $e$.
\end{proof}

\begin{dfn} Let $A$ be a finite-dimensional algebra. We call $A$ \emph{torsionless-finite} if the regular module $A$ is cogen-finite, or equivalently \cite[Cor.~1]{RingelSelectedTopics} if $\kdual A$ is gen-finite. 
\end{dfn}

Examples of torsionless-finite algebras include hereditary and concealed algebras; see for example Oppermann's survey \cite[Ex.~1.13]{O}. If $\rad^n(A)=0$ and $A/\rad^{n-1}(A)$ is representation-finite, then $A$ is torsionless-finite \cite[Special cases (1)]{RingelSelectedTopics}. In particular, if $\rad^2(A)=0$ then $A$ is torsionless-finite.

\section{Example}
\label{s:example}
\subsection{\texorpdfstring{A module for the $n$-subspace quiver}{A module for the n-subspace quiver}}
Let $A$ be the path algebra of the $n$-subspace quiver:

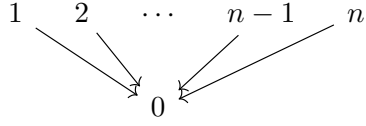
\[\begin{tikzcd}[column sep=10pt]
1\arrow{drr}&2\arrow{dr}&\cdots&n-1\arrow{dl}&n\arrow{dll}\\
&&0
\end{tikzcd}\]
When treating an $A$-module $X$ as a representation of this quiver, we denote by $X_i$ the linear map carried by the arrow $i\to 0$. We fix the $A$-module 
\[M= S(0)\oplus \kdual A = \Big(\bigoplus_{i=0}^n S(i)\Big)\oplus Q(0),\]
where $S(i)$ denotes the simple module supported at a vertex $i$ and $Q(i)$ its minimal injective envelope. 
It follows from Corollary~\ref{gen-fin-quiver} that $M$ is gen-finite, and indeed we may compute $\gen(M)=\add{M}$.
 
First, we calculate the orbit closure $\clorbit{M} \subseteq \rep{A}{\begin{smallmatrix}  
2 & 2 & \cdots & 2 & 2\\ &&2 &&
\end{smallmatrix}} $. Since $\gen (M) = \add{M}$ it follows from Theorem~\ref{rankvar=clorbit} that
\[ 
\begin{aligned}
\clorbit{M} & =\{ N \in  \rep{A}{\begin{smallmatrix}  
2 & 2 & \cdots & 2 & 2\\ &&2 &&
\end{smallmatrix}} \mid\text{$\left[ N, Y\right] \geq \left[ M, Y \right]$ for all $Y\in \add{M}$}\}\\
&= \{ N \in   \rep{A}{\begin{smallmatrix}  
2 & 2 & \cdots & 2 & 2\\ &&2 &&
\end{smallmatrix}} \mid \left[ N, S(0) \right] \geq \left[ M, S(0) \right]=1 \}\\
&\cong  \left\{ \left( \begin{smallmatrix} a_i & b_i \\ c_i & d_i \end{smallmatrix} \right)_{1\leq i \leq n} \in \Mat_{2\times 2} (\K )^n \mid \rk\left( \begin{smallmatrix} 
a_1 & b_1 & a_2 & b_2 & \cdots & a_n & b_n \\
c_1 & d_1 & c_2 & d_2 & \cdots & c_n & d_n 
\end{smallmatrix}\right) \leq 1 \right\} \\
&\cong V(X_iY_j-X_jY_i,\ i\neq j ) \subseteq \Spec \K[ X_1, \ldots , X_{2n}, Y_1, \ldots , Y_{2n}].
\end{aligned}
\]
For the third step, note that maps $N_i\colon \K^2\to \K^2$ for $1\leq i \leq n$ determine a module $N$ with $[N,S(0)]\geq1$ if and only if there is a non-zero vector $(x,y)$ such that $(x,y) (N_1, N_2, \ldots , N_n)=0$, this being equivalent to the rank inequality. 
Thus $\clorbit{M}$ is a determinantal variety. We also have $\clorbit{M}\cong\clorbit{\widetilde{M}}$, where $\widetilde{M}$ is the representation
\[\K^{2n}\xrightarrow{\left(\begin{smallmatrix}1&0&1&0&\cdots&1&0\\0&0&0&0&\cdots&0&0\end{smallmatrix}\right)}\K^2\]
of the $\mathsf{A}_2$-quiver, and hence $\clorbit{M}$ is a normal and Cohen--Macaulay variety \cite{BZ3}.

Next we calculate a quiver Grassmannian coming from $M$, and its irreducible components. Choosing bases, we identify $M$ as the point of $\rep{A}{\begin{smallmatrix}  
2 & 2 & \cdots & 2 & 2\\ &&2 &&
\end{smallmatrix}}$ given by $M_i = \left( \begin{smallmatrix} 1 & 0 \\ 0 & 0 \end{smallmatrix} \right)$ for all $1 \leq i \leq n$.  
If $d$ is the dimension vector given by $1$ at every vertex, the quiver Grassmannian $\Gra{A}{M}{d}$ is
\begin{align*}
\Gra{A}{M}{d} &=\{ (L_0,L_1, L_2, \ldots , L_n ) \in \mathbb{P}^1 \times \cdots \times \mathbb{P}^1\mid M_i (L_i ) \subseteq L_0,\ 1\leq i \leq n\}\\
&=\{(L_0,[0:1],\dotsc,[0:1])\mid L_0\in\P^1\}\cup\{([1:0],L_1,\dotsc,L_n)\mid L_i\in\P^1,\ 1\leq i\leq n\}.
\end{align*}
Let $f\colon(\P^1)^n\to\Gra{A}{M}{d}$ be the regular map
\[f(t_1,\dotsc,t_n)=([1:0],t_1,\dotsc,t_n)\]
and $g\colon\P^1\to\Gra{A}{M}{d}$ the regular map
\[g(t)=(t,[0:1],\dotsc,[0:1]).\]
The images of these maps are closed and irreducible, cover $\Gra{A}{M}{d}$, and neither is contained in the other, so they are the irreducible components. If $U\in\Bild{f}$, then we calculate directly that $M/U\cong S:=\bigoplus_{i=0}^nS(i)$, so $\Bild{f}=\grstrat{S}=\clgrstrat{S}$ is one irreducible component, isomorphic to $(\P^1)^n$. Similarly, $M/g(t)\cong Q$ for $t\ne[1:0]$, whereas $M/g([1:0])\cong S$. Thus the other irreducible component is $\Bild{g}=\grstrat{Q}\sqcup\{U_0\}=\clgrstrat{Q}\cong\P^1$, where $U_0=g([1:0])$ is the unique intersection point of $\clgrstrat{S}$ and $\clgrstrat{Q}$. In particular, $U_0$ is the only singular point of $\Gra{A}{M}{d}$.

\subsection{The cogenerator-tilted algebra}
To construct our desingularisations we take $E=M$, noting that $M$ is a cogenerator with  $\add{M}=\gen(M)$. Then $\Gamma = \End_A(E)^{\op}\cong \K Q_{\Gamma}/ \rad^2(Q_\Gamma)$ for $Q_{\Gamma}$ the quiver
\[\begin{tikzcd}[column sep=10pt]
1'\arrow{drr}&2'\arrow{dr}&\cdots&(n-1)'\arrow{dl}&n'\arrow{dll}\\
&&0'\arrow{d}\\
&&\bullet
\end{tikzcd}\]
Here each vertex $i'$ corresponds to the summand $Q(i)$ of $M$, noting that $Q(i)=S(i)$ for $i\geq1$, and $\bullet$ corresponds to $S(0)$.

The projective $\Gamma$-module $P=\Hom_A(E, \kdual A)= \bigoplus_{i=0}^n P(i')$ is faithful. Since $\left[S(\bullet), P(i')\right]=0$ for $1\leq i \leq n$, a minimal left $\add{P}$-approximation of $S(\bullet)$ is given by a monomorphism $S(\bullet)\to P(0')$, with cokernel $S(0')$. This implies that $T_P= P \oplus S(0')$ is the $P$-special tilting $\Gamma $-module. We calculate that the cogenerator-tilted algebra $B= \End_{\Gamma} (T_P)^{\op}$ of $E$ is isomorphic to the path algebra $\K Q_B$ for $Q_B$ the quiver
\[\begin{tikzcd}[column sep=10pt]
{[1]}\arrow{drr}&{[2]}\arrow{dr}&\cdots&{[n-1]}\arrow{dl}&{[n]}\arrow{dll}\\
&&{\circ}\arrow{d}\\
&&{[0]}
\end{tikzcd}\]
The vertex $[i]$ corresponds to the summand $P(i')$ of $T_P$, and $\circ$ to the summand $S(0')=\Omega^{-1}S(\bullet)$. The special idempotent is $e:= \sum_{i=0}^ne_{[i]}$, corresponding to the summand $P$ of $T_P$, and we can check that $eBe \cong A$ as expected. Set $C:= c(M)$, where $c$ is the intermediate extension corresponding to $e$. Since $c$ maps simples to simples \cite[\S4]{Ku} and injectives to injectives by Theorem~\ref{int-ext-calcs}(2), we may calculate
\[ 
c(M) = c(S) \oplus c(Q) = \Big(\bigoplus_{i=0}^n S[i]\Big) \oplus Q[0],   
\]  
and so $\Dim c(M) = \left( \begin{smallmatrix} 2 & 2 & \cdots & 2 & 2 \\
&&1&& \\
&&2&&
\end{smallmatrix} \right)$. 

\subsection{Desingularisation of the orbit closure}
To desingularise $\clorbit{M}$, we are interested in the stable $B$-modules, which are the modules in 
$\cogen (\widetilde{Q})$ for $\widetilde{Q} = \kdual(eB)=\bigoplus_{i=1}^n S[i] \oplus Q[0]$. These are the modules with socle supported away from the vertex $\circ$, or equivalently, in the language of quiver representations, those for which the arrow $\circ \to [0]$ carries a monomorphism. It follows that
\[ 
X:=\rep{B}{\begin{smallmatrix} 2 & 2 & \cdots & 2 & 2 \\
&&1&& \\
&&2&&
\end{smallmatrix}}^{\st} / \Gl_1 = \{ (N, U) \in \rep{A}{d} \times \P^1 \mid \Bild N_i \subseteq U,\ 1 \leq i \leq n\}, 
\]
where, under this identification, $U$ is the image of the monomorphism on the arrow $\circ\to[0]$.

We can check that $X$ is smooth and irreducible by considering the projection $\pr_2 \colon 
X  \to \P^1$. The fibre over $[1:0]$ consists of all tuples $(N_1, \ldots ,N_n) \in \Mat_{2\times 2} (\K )^n$  such that each $N_i$ has lower row zero, and hence this fibre is an affine space (of dimension $2n$). In fact, it is a $\mathbf{B}$-representation where $\mathbf{B}\subseteq \Gl_2 (\K )$ denotes the upper triangular matrices operating by conjugation. Since $\pr_2 $ is a $\Gl_2$-equivariant map into the homogeneous space $\P^1$, it follows that $X$ is is a vector bundle over $\P^1$, and so is smooth and irreducible. In particular, since $c(M)$ is rigid, $X=\clorbit{c(M)}^\st$. Thus the desingularisation of $\clorbit{M}$ from Theorem~\ref{gen-finite-desing} is
\[ 
\pi = \pr_1 \colon \{ (N, U) \in \rep{A}{d} \times \P^1 \mid \Bild N_i \subseteq U,\ 1 \leq i \leq n\} \to \clorbit{M}. 
\]

\subsection{Desingularisation of the quiver Grassmannian}
Now we describe our desingularisation of $\Gra{A}{M}{d}$, recalling that $d$ has value $1$ at every vertex.
Let
\[d_Q:= \Dim c(M) - \Dim c(Q) =  \left( \begin{smallmatrix}1 & 1 & \cdots & 1 & 1 \\
&&0&& \\
&&1&&
\end{smallmatrix} \right)\]
and
\[d_S := \Dim c(M) - \Dim c(S) = \left( \begin{smallmatrix}1 & 1 & \cdots & 1 & 1 \\
&&1&& \\
&&1&&
\end{smallmatrix} \right).\]

Every module of dimension vector $d_Q$ is isomorphic to $c(S)$, so (see \cite[Thm.~4.2]{R}) the Grassmannian  $\Gr_B\binom{c(M)}{d_Q}=\mcS_{[c(S)]}$ is smooth and irreducible of dimension $\left[c(S), c(M)\right] - \left[c(S), c(S)\right] = \left[c(S), c(Q)\right] = \left[ S, Q\right] = 1$, and hence it coincides with $\clgrstrat{c(Q)}$. Similarly, by considering the quotients, $\Gra{B}{c(M)}{d_S}=\grstrat{c(S)}$ is smooth and irreducible of dimension $\left[c(M), c(S)\right] - \left[c(S), c(S)\right] = \left[Q,S\right] =n$, and so in particular $\clgrstrat{c(S)}=\Gra{B}{c(M)}{d_S}$. Thus the desingularisation given by Corollary~\ref{Gr-desing} coincides with the na\"ive desingularisation $p\colon \Gra{B}{c(M)}{d_Q} \sqcup \Gra{B}{c(M)}{d_S} \to \Gra{A}{M}{d}$, obtained by taking the disjoint union of the two irreducible components (cf.\ \cite[\S5]{KS2}).

\section*{Acknowledgements}
We would like to thank Teresa Conde and William Crawley-Boevey for helpful comments. These thanks are extended to the anonymous referee for useful suggestions, such as the addition of Example~\ref{KP-eg}. The first author also thanks Henning Krause for an invitation to visit Bielefeld, funded by the Deutsche Forschungsgemeinschaft grant SFB 701, from which this project resulted. Subsequently he was supported by a postdoctoral fellowship from the Max-Planck-Gesellschaft. The second author is supported by the Alexander von Humboldt-Stiftung in the framework of an Alexander von Humboldt Professorship endowed by the Federal Ministry of Education and Research.

\bibliography{biblio}

\end{document}